\documentclass[11pt,twoside,reqno]{article}
\usepackage{mathrsfs,subfig,latexsym,array,multirow}
\usepackage{amsmath,amssymb,amsthm,datetime,graphicx}
\usepackage{color, amsbsy}

\DeclareMathAlphabet{\mathpzc}{OT1}{pzc}{m}{it}
\title{
Spurious pressure in Scott-Vogelius  elements
}
\author{
Chunjae Park
\thanks{Department of Mathematics,
Konkuk University, Seoul 05029, Korea, 
\hspace{1mm} Email: cjpark@konkuk.ac.kr}}

\begin{document}

\newtheorem{theorem}{Theorem}[section]
\newtheorem{remark}[theorem]{Remark}
\newtheorem{lemma}[theorem]{Lemma}
\newtheorem{proposition}[theorem]{Proposition}
\newtheorem{definition}[theorem]{Definition}
\newtheorem{assumption}{Assumption}[section]

\def\R{\mathbb R}
\def\O{\Omega}
\def\p{\partial}
\def\Th{\mathcal{T}_h}
\def\Cs{C_{\sigma}}
\def\vts{\vartheta_{\sigma}}
\def\disp{\displaystyle}
\def\div{\mathrm{div}\hspace{0.5mm}}
\def\curl{\mathbf{curl}\hspace{0.5mm}}

\def\alp{\alpha}
\def\bet{\beta}
\def\gam{\gamma}
\def\del{\delta}
\def\lam{\lambda}

\def\m{\mathpzc{m}}
\def\n{\mathbf{n}}
\def\t{\boldsymbol{\tau}}
\def\u{\mathbf{u}}
\def\v{\mathbf{v}}
\def\w{\mathbf{w}}
\def\x{\mathbf{x}}
\def\y{\mathbf{y}}
\def\z{\mathbf{z}}
\def\B{\mathcal{B}}
\def\C{\mathbf{C}}
\def\M{\mathbf{M}}
\def\G{\mathbf{G}}
\def\P{\mathbf{P}}
\def\V{\mathbf{V}}
\def\W{\mathbf{W}}
\def\X{\mathbf{X}}
\def\Y{\mathbf{Y}}

\def\bxi{\boldsymbol{\xi}}
\def\st{s}
\def\ib{g}

\def\sg{\mathpzc{s}}
\def\Sg{\mathpzc{S}}
\def\f{\mathpzc{e}}
\def\ii{\mathbf{i}}
\def\th{\theta}
\def\dh{{\theta}}
\def\hp{\widehat{p}}
\def\he{\widehat{e}}
\def\tp{\widetilde{p}}

\def\SS{\mathcal{S}}
\def\NN{\mathcal{N}}
\def\VV{\mathcal{V}}
\def\EE{\mathcal{E}}
\def\upnorm{|\mathbf{u}|_5 + |p|_4}
\def\AreaV{B(\V)}
\def\VVone{|\overline{\V\V_1}|}
\def\VVtwo{|\overline{\V\V_2}|}
\def\Q{\mathpzc{Q}}
\def\bR{\mathbf{R}}

\def\alpt{\widetilde{\alpha}}
\def\bett{\widetilde{\beta}}
\def\alpe{\hat{\alp}}
\def\bete{\hat{\bet}}
\def\te{\tilde{e}}
\def\aa{a}
\def\bb{b}
\def\pskip{\hspace{1pt}}
\def\mmskip{\vspace{1mm}}
\def\osh{\overline{s_h}}

\newcommand {\snorm}[2] {| #1 |_{#2}}
\newcommand {\norm}[2] {\| #1 \|_{#2}}
\newcommand{\Evec}[2]{\overrightarrow{#1#2}}
\newcommand{\vertiii}[1]{{\left\vert\kern-0.25ex\left\vert\kern-0.25ex\left\vert #1 
        \right\vert\kern-0.25ex\right\vert\kern-0.25ex\right\vert}}

\maketitle
\begin{abstract}
  We will analyze the characteristics of Scott-Vogelius finite elements
  on singular vertices, which cause spurious pressures
  on solving Stokes equations. A simple  postprocessing
  will be suggested to remove those spurious pressures.
\end{abstract}

\section{Introduction}
The Scott-Vogelius element is the typical high order finite element space
which can be applied to solve Stokes problems.
Its inf-sup condition was proved in several ways,
only when the triangulation has no singular vertex
\cite{Falk2013, Guzman2018, Scott1985}. While it struggles with singular vertices,
the inf-sup constant $\beta$ is not proper even in case of nearly singular vertices.

In practice,
when the mesh has a nearly singular vertex,
the discrete solution in pressure shows an error
which is improper at a glance as in Figure \ref{fig:pandph}
in the numerical test section.
In this paper, we will call it spurious and analyze its causes.
%to remove it by simple postprocessing. 

The punchline of the paper is splitting of the error in stable and unstable parts
on nearly singular vertices.
We will suggest a simple postprocessing to remove the unstable parts
from the discrete pressure obtained by the standard finite element methods.
The suggested postprocessing could improve the error
even in case of regular vertices.

In our analysis, a cubic polynomial depicted in Figure \ref{fig:spu_ref}
plays a key role with its interesting quadrature rule.
Spurious pressures consist of those polynomials 
 at singular or nearly singular vertices.
Although, in this paper, we deal with only the Scott-Vogelius elements
of the lowest order
in two dimensional domains, we might start its extension to general order
if we find such a polynomial there. 

For three dimensional Scott-Vogelius elements,
the general extension
 identifying singular vertices and edges
 is  still on its way, in spite of some results on it 
 \cite{Neilan2015, Zhang2007, Zhang2011}.

The paper is organized as follows.
In the next two sections, the quasi singular vertices and Scott-Vogelius elements
will be introduced.
In section \ref{sec:spu}, we will show that the discrete Stokes problem is singular
due to the presence of spurious pressures,
if the mesh has exactly singular vertices.
In case of quasi singular vertices, the spurious component of the error in pressure
will be identified in  section  \ref{sec:errprs}
utilizing a new basis of pressure designed in section \ref{sec:basis}.
Then, we will  devote section \ref{sec:remove} to removing the spurious error
from the discrete pressure.
Finally, some numerical tests will be presented in the last section.

Throughout the paper,
$|\x|$ denotes an area or length if $\x$ is a triangle, edge or vector
and
$\# S$ does the cardinality of a set $S$.

 \section{Quasi singular vertex}
Let $\O$ be a connected polygonal domain in $\R^2$ and 
 $\{\Th\}_{h>0}$ a regular family of triangulations of $\O$
with a shape regularity parameter $\sigma>0$.
Denote by $\VV_h, \EE_h$,
the sets of all vertices and edges in $\Th$, respectively.
If a vertex $\V\in\VV_h$ belongs to $\p\O$, we call it a boundary vertex,
otherwise, an interior vertex. Similarly, an edge $E\in\EE_h$ is called
a boundary edge if $E\subset \p\O$,
otherwise, an interior edge.

A vertex $\V\in\VV_h$ is called singular or exactly singular
if two lines are enough to cover
all edges sharing $\V$ as in Figure \ref{fig:extsingvtx}.
\begin{figure}[ht]
  \hspace{3cm}
\includegraphics[width=0.6\linewidth]{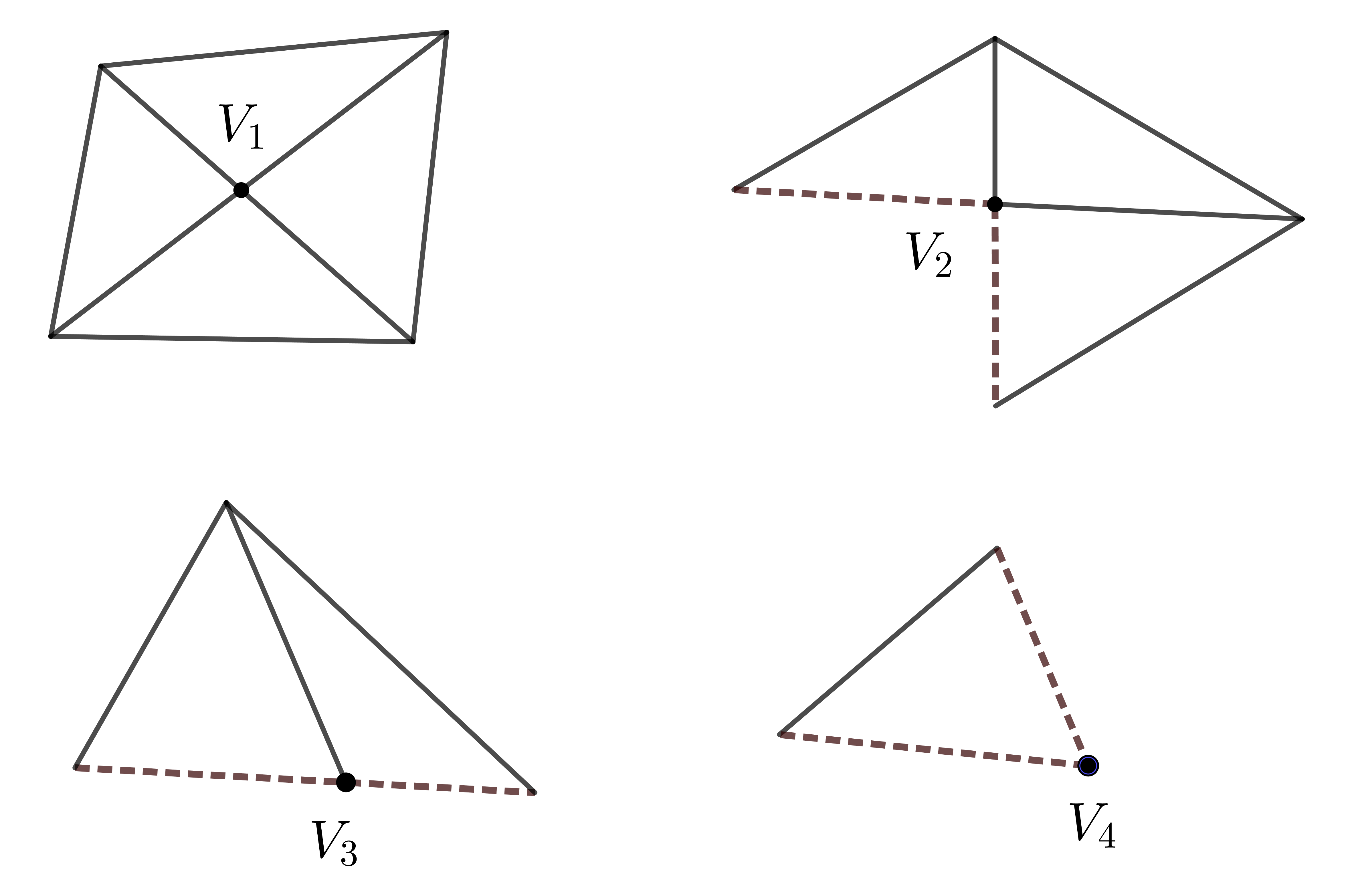}
\caption{Four types of exactly singular vertices $\V_1, \V_2, \V_3, \V_4$
  (dashed edges belong to $\p\O$.)}
\label{fig:extsingvtx}
\end{figure}
For each vertex $\V$, denote by
$\Upsilon(\V)$, the set of all sums of two adjacent angles of $\V$
in two back-to back triangles in $\Th$.
Then $\Upsilon(\V)=\{\pi\} \mbox{ or } \emptyset$ if and only if $\V$ is singular.
For examples, in Figure \ref{fig:extsingvtx},
$$\Upsilon(\V_1)=\Upsilon(\V_2)=\Upsilon(\V_3)=\{\pi\},\quad \Upsilon(\V_4)=\emptyset.$$

Since $\{\Th\}_{h>0}$ is regular, there exists $\vartheta>0$ such that
\[ \vartheta=\inf\{ \theta\ |\  \theta \mbox{ is an angle of a triangle } K\in\Th,
  h>0  \}.  \]
Set
\begin{equation}\label{def:vts}
  \vartheta_{\sigma} =\min(\vartheta, \pi/6),
\end{equation}
then $\vts$ depends on
the shape regularity parameter  $\sigma$ of $\{\Th\}_{h>0}$. 
From \eqref{def:vts},
we note that every angles $\theta$ of a triangle $K$ in $\Th$ satisfies that 
\begin{equation}\label{def:vts2}
 \vts \le  \theta \le  \pi-2\vts.
\end{equation}

We will call a vertex $\V\in\VV_h$ quasi singular
if it is singular or nearly singular.
For quantification, define a set 
\begin{equation}\label{def:sing}
  \mathcal{S}_h=\{ \V\in\VV_h\ :\ 
  |\Theta -\pi| < \vts \mbox{ for all } \Theta\in\Upsilon(\V)\}.
\end{equation}
Then, 
we call a vertex $\V$ quasi singular if $\V\in\SS_h$, otherwise regular.
In Figure \ref{fig:singvtx}, examples of quasi but not exactly singular vertices
are depicted.
Interior quasi singular vertices are slight perturbations
of exactly singular ones. It results in the following lemma:
\begin{figure}[ht]
  \hspace{3cm}
\includegraphics[width=0.6\linewidth]{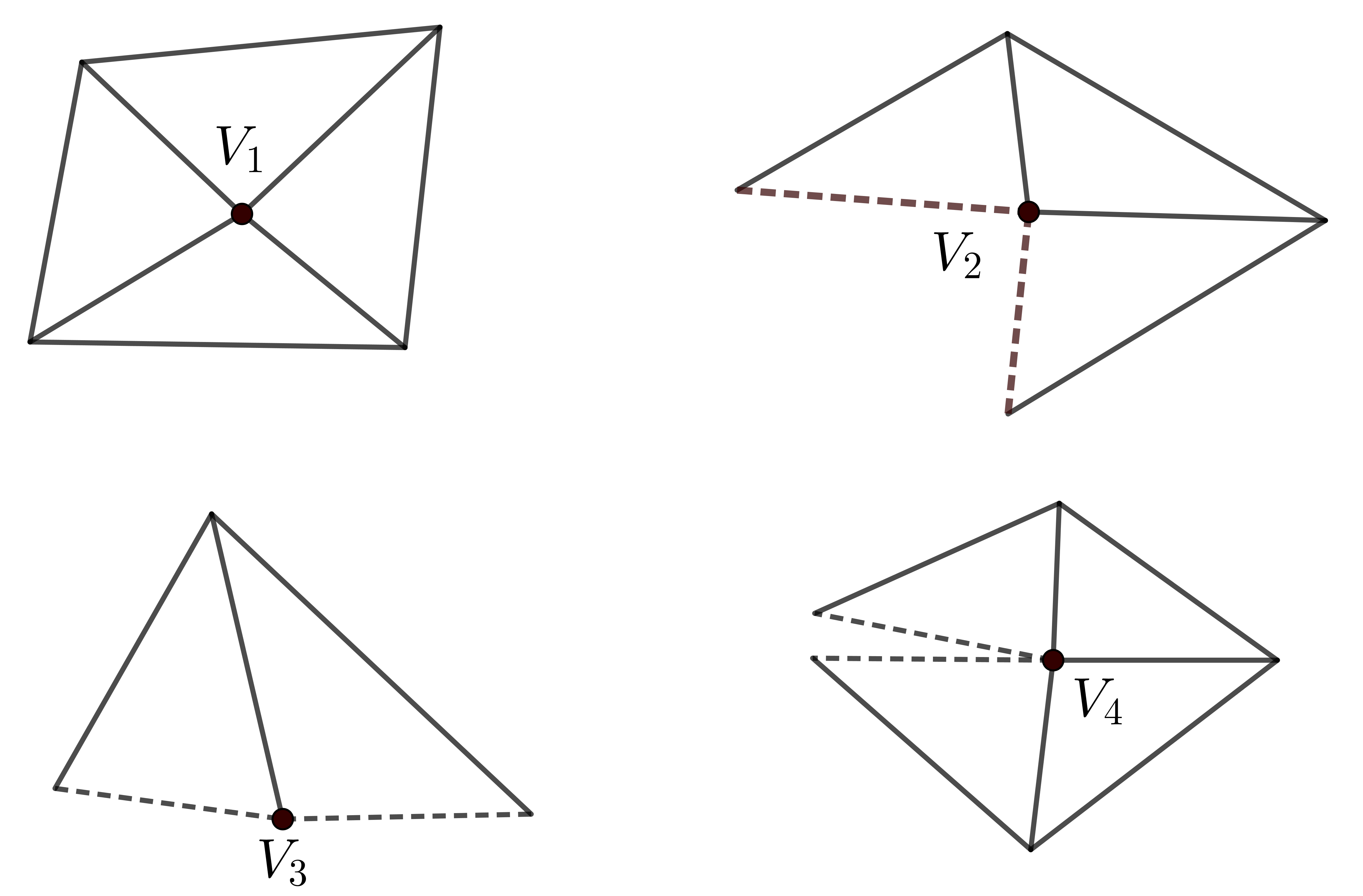}
\caption{Quasi but not exactly singular vertices $\V_1, \V_2, \V_3, \V_4$
  (dashed edges belong to $\p\O$.)}
\label{fig:singvtx}
\end{figure}
\begin{lemma}\label{lem:insing4}
  If  $\V$ is an interior quasi singular vertex,
  then the number of all triangles sharing $\V$ is $4$.
\end{lemma}
\begin{proof}
  Let $N$ be the number of all triangles sharing $\V$ and
  $\theta_1, \theta_2, \cdots, \theta_N$ back-to-back angles of $\V$.
Set
  \[ \Theta=\min\{ \theta_1 + \theta_2, \theta_2 + \theta_3, \cdots,
    \theta_N + \theta_1\}.\]
  Then,
  \begin{equation}\label{eq:NThetaV}
    N\Theta \le 2 \sum_{i=1}^N \theta_i = 4\pi.
  \end{equation}
  If $N\ge 5$, then \eqref{def:sing} and \eqref{eq:NThetaV}
  makes the following contradiction to $\vts \le \pi/6$
  in \eqref{def:vts}:
 \[ \pi -\vts < \Theta \le \frac45 \pi.\] 
If $N=3$, we have  from \eqref{def:vts2},  
  \[ \theta_1+ \theta_2 = 2\pi-\theta_3 \ge 2\pi -(\pi-2\vts) = \pi+2\vts. \]
It contradicts to $\V\in\SS_h$.
\end{proof}

Each interior quasi singular vertex in $\SS_h$ is isolated from
others in $\SS_h$ in the sense of the following lemma.
\begin{lemma}\label{lem:isolvtx}
  There is no interior edge connecting two quasi singular vertices in $\SS_h$.
\end{lemma}
\begin{proof}
  Let $E$ be an interior edge whose two endpoints  $\V_1, \V_2$
  are quasi singular in $\SS_h$.
  Then, there exist two triangles sharing $E, \V_1, \V_2$
  as in Figure \ref{fig:sep2ndcls}.

  Consider the quadrilateral $Q$ whose vertices are $\V_1, \V_4, \V_2, \V_3$ and
  one of its diagonals is $E$.
  Denote the angle of $\V_i$ in $Q$  by $\theta_i$, $i=1,2,3,4$.
Then,  from \eqref{def:vts2} and
  the definition of $\SS_h$, we have
  \[ \pi -\theta_j < \vts ,\mbox{ if } j=1,2, \qquad
 \vts\le \theta_j,\mbox{ if } j=3,4.  \]
 It meets with the following contradiction:
   \[ 2\pi < \theta_1 + \theta_2+ \theta_3+ \theta_4 =2\pi.\]
\begin{figure}[ht]
\hspace{3cm}
\includegraphics[width=0.6\textwidth]{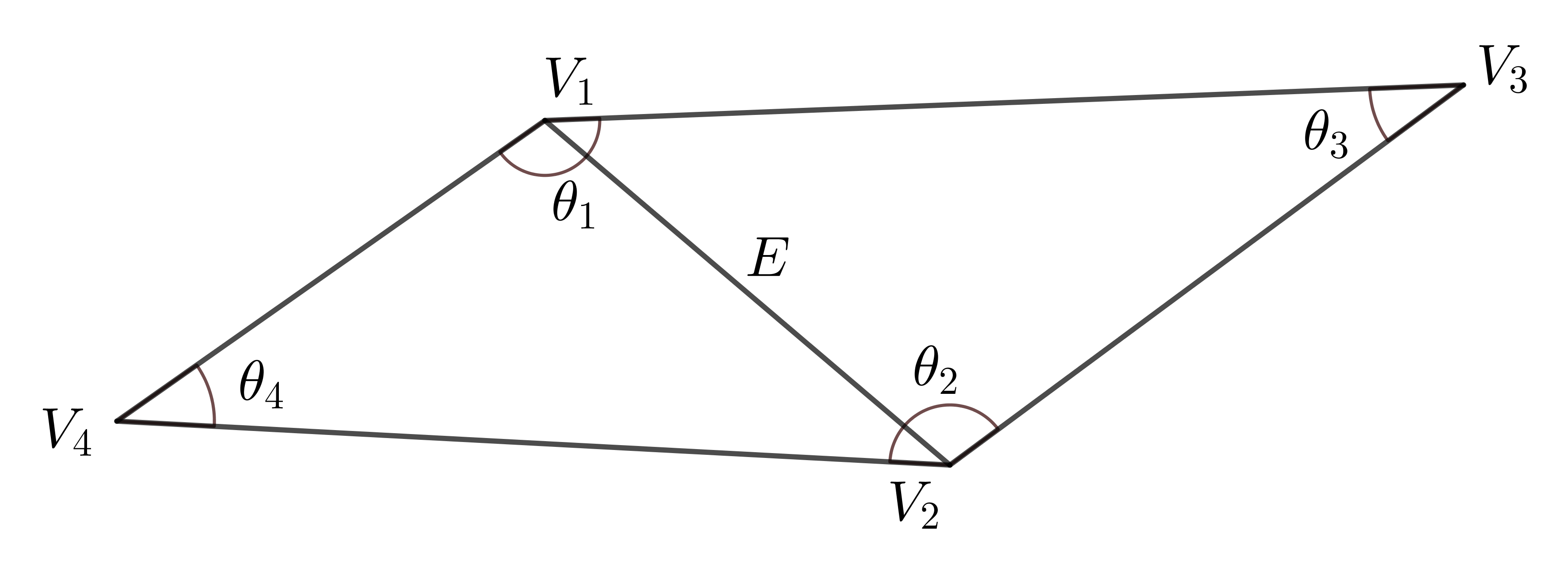}
\caption{Two quasi singular vertices $\V_1, \V_2$ form a quadrilateral
  with sharp angles $\theta_3, \theta_4$}\label{fig:sep2ndcls}
\end{figure}
\end{proof}

\section{Scott-Vogelius elements}\label{sec:SV}

Let's define the discrete polynomial spaces $ \mathcal{P}_{k,h}(\O)$ as
\[ \mathcal{P}_{k,h}(\O) =\{ v_h \in L^2(\O)\ : \ v_h|_K \in P^k
\mbox{ for all triangles } K\in\Th \}, \quad  k \ge 0.\] 
Then the Scott-Vogelius finite element space is the pair of $X_h^k, M_h^{k-1}$
such that
\[ X_h^k= {[\mathcal{P}_{k,h}(\O)\cap H_0^1(\O)]^2},\qquad
  M_h^{k-1}={\mathcal{P}_{k-1,h}(\O)\cap L_0^2(\O)},\quad k\ge 4,\]
where $L_0^2(\O)$ is the space of square integrable functions whose means vanish.
In this paper, we deal with only the  Scott-Vogelius finite element space
of the lowest order:
\[ X_h= {[\mathcal{P}_{4,h}(\O)\cap H_0^1(\O)]^2},\qquad
  M_h={\mathcal{P}_{3,h}(\O)\cap L_0^2(\O)}).\]

The incompressible Stokes problem is to 
find $(\u,p)\in [H_0^1(\O)]^2 \times L_0^2(\O)$ such that
\begin{equation}\label{prob:conti}
( \nabla \u,\nabla \v) +(p,\div \v)+(q,\div \u)
  = (\mathbf{f},\v) \quad \mbox{ for all } (\v,q) \in [H_0^1(\O)]^2\times  L_0^2(\O),
\end{equation}
 for a given source function $\mathbf{f}\in [L_0^2(\O)]^2$.
We will consider 
the discrete Stokes problem for \eqref{prob:conti} to find
$(\u_h,p_h)\in X_h \times M_h$ such that
\begin{equation}\label{prob:discr}
 ( \nabla \u_h,\nabla \v_h) + (p_h, \div \v_h)+(q_h, \div \u_h)
  = (\mathbf{f},\v_h)\quad \mbox{ for all } (\v_h,q_h) \in X_h\times M_h.
\end{equation}

\subsection{Error in velocity}
Let $M_h^S$ is the space of spurious pressures such that
\begin{equation}\label{def:Mhs}
  M_h^S =\{ s_h \in M_h\ | \ (s_h, \div\v_h)=0 \mbox{ for all } \v_h\in X_h\}.
\end{equation}
Unfortunately, $M_h^S$ is not null, if $\Th$ has an exact singular vertex
as will be discussed in subsection \ref{sec:example_spu} below.
The discrete problem \eqref{prob:discr}, however, has at least one solution,
 even if $M_h^S\neq \{0\}$. 
\begin{lemma}
  There exists $(\u_h,p_h)\in X_h \times M_h$ satisfying \eqref{prob:discr}.
  In addition, $\u_h$ is unique.
\end{lemma}
\begin{proof}
  Let $M_h = M_h^S \bigoplus \hat{M_h}$ for some
subspace $\hat{M_h}$. Then there exists a
  unique $(\u_h,\hat{p_h})\in X_h \times \hat{M_h}$ 
satisfying \eqref{prob:discr}, since the discrete problem
  is not singular on $X_h \times \hat{M_h}$.
\end{proof}

Let $\Upsilon'(\V)=\Upsilon(\V)\cup\{0\}$
and define a parameter $\Theta_{\min}$ of the triangulation $\Th$ as
\[ \Theta_{\min}  = \min_{\V\in\VV_h} 
  \max_{\Theta\in\Upsilon'(\V)}|\sin\Theta|. 
\]
The following inf-sup condition is well known \cite{Guzman2018}:
\begin{equation}\label{cond:infsup}
 \Theta_{\min} \beta  \| q_h \|_0 \le
  \sup_{\v_h\in X_h\setminus\{0\}}\frac{(q_h,\div\v_h)}{|\v_h|_1},
    \quad \forall q_h\in M_h.
  \end{equation}
If $\Th$ has a quasi singular vertex in $\SS_h$, 
$\Theta_{\min}$ is zero or might be quite small. 
It could spoil the discrete pressure $p_h$ as in Figure \ref{fig:pandph}.
Although the inf-sup condition in \eqref{cond:infsup} depends on $\Theta_{\min}$, the error in velocity
is stable independently of  $\Theta_{\min}$.

Throughout inequalities in the paper,
a generic notation $C$ denotes a constant which depends only on $\O$ and
the shape regularity parameter $\sigma$.

\begin{theorem}\label{thm:veloerror}
    Let
    $(\u,p)\in [H_0^1(\O)]^2\times L_0^2(\O)$
    and
    $(\u_h, p_h)\in  X_h\times M_h$ satisfy \eqref{prob:conti}, \eqref{prob:discr}, respectively. Then, if $\u\in[H^5(\O)]^2$, we have
    \[ |\u -\u_h|_1 \le C h^4 |\u|_5.\]
   \end{theorem}
   \begin{proof}
     Since $\div\u=0$, there exists  a stream function $\phi\in H^6(\O)$ of $\u$
     which is constant on each component of $\p\O$.
     Let $\phi_h$ be the projection of $\phi$ into the space of
     $C^1$-Argyris triangle elements which are locally $P^5$
     \cite{Braess2001, Brenner2002, Ciarlet}.
     Then, $\nabla \phi_h$ is continuous in $\O$ and vanishes on $\p\O$ and
     $\phi_h$ satisfies that 
     \[ |\phi-\phi_h|_2 \le C h^4 |\phi|_6.\]
     Thus, if we define $\Pi_h \u= \curl \phi_h$, we have $\Pi_h\u\in X_h$ and
     \begin{equation}\label{eq:uPiuerror}
       |\u - \Pi_h \u|_1  \le C h^4 |\u|_5.
     \end{equation}
     
     Let
   \[ V_h =\{ \v_h\in X_h\ | \ (q_h,\div \v_h) =0 \mbox{ for all } q_h\in M_h\}.\]
Note $\div\v_h=0$, if $\v_h\in V_h$.   Then, from \eqref{prob:conti}, \eqref{prob:discr}, $\u$ and $\u_h$
   satisfy that
   \begin{equation}\label{eq:nabluuhvh}
     (\nabla \u -\nabla\Pi_h\u, \nabla\v_h)=(\nabla\u_h-\nabla\Pi_h\u,\nabla\v_h) \mbox{ for all } \v_h\in V_h.
   \end{equation}
   Since $\u_h, \Pi_h\u_h\in V_h$,
   we have, for $\v_h=\u_h- \Pi_h\u\in V_h$ in \eqref{eq:nabluuhvh}, 
 \[ \snorm{\u_h-\Pi_h\u}{1}^2 \le   \snorm{\u-\Pi_h\u}{1} \snorm{\u_h-\Pi_h\u}{1}.\]
 It completes the proof with \eqref{eq:uPiuerror}.  
  \end{proof}   

\section{Spurious pressure}\label{sec:spu}
\subsection{Sting functions}
Let $K$ be a triangle in $\Th$ which has an edge $E$ and its opposite vertex $\V$
as in  Figure \ref{fig:svze}-(b).
Denote by $\lam(\x)$, a barycentric coordinate of $\x$ vanishing on $E$ such that
\[ \lam(\x)=(-\n)\cdot(\x-\M),\]
where  $\n$ is the unit outward normal vector of $K$ on $E$ and
$\M$ is the center of $E$.

With a specific function $\f$:
\begin{equation}\label{def:spucoef}
 \f(t)= \frac1{10}(56 t^3 -63t^2 +18t-1),
\end{equation}
define a cubic polynomial $\st_{E\V} \in P^3(K)$ determined by the edge $E$
and its opposite vertex $\V$:
\begin{equation}\label{def:zEV}
  \st_{E\V} (\x) = \f\Big(\frac{\lam(\x)}{H}\Big),
\end{equation}
where $H$ is the distance between $E$ and $\V$.
A graph of $\st_{E\V}$ is depicted in Figure \ref{fig:spu_ref}
in the reference triangle $\widehat{K}$.
We would name $\st_{E\V}$ a sting function after its look.

In the remaining of the paper,
a local function such as $\st_{E\V}$ defined on $K$
is identified with its trivial extension on $\O$ vanishing outside $K$.
We also use a notation $\Cs$ for a generic constant which depends only on
the shape regularity parameter $\sigma$.

\begin{figure}[t]
\hspace{3cm}
\includegraphics[width=0.6\textwidth]{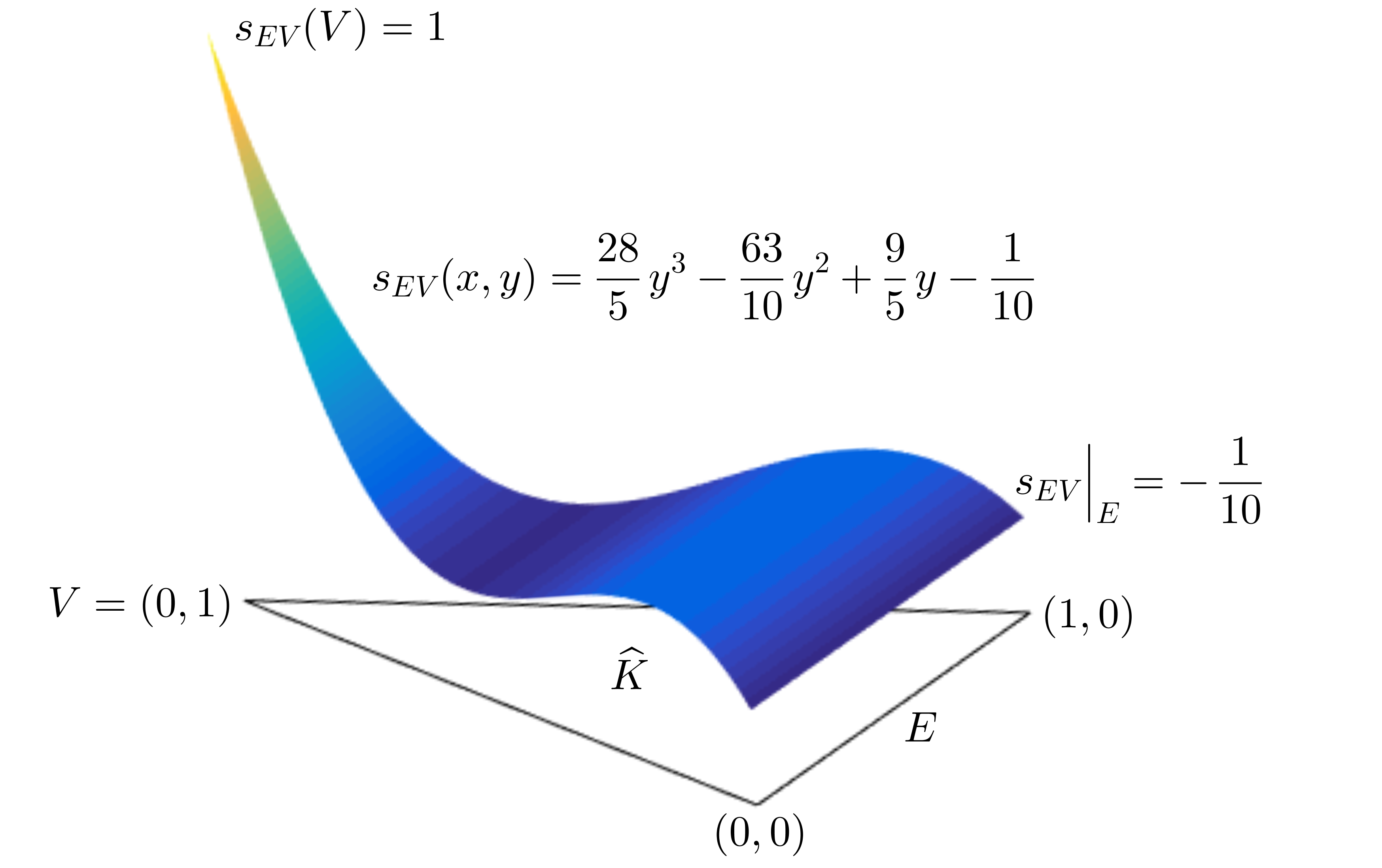}
\caption{ A sting function $\st_{E\V}\in P^3$ over $\widehat{K}$
}\label{fig:spu_ref}
\end{figure}

\subsection{ Quadrature rules}
The choice of $\f$ in \eqref{def:spucoef} makes
the sting function $\st_{E\V}$ satisfy the following two quadrature rules 
which play key roles in our error analysis for pressure. 
\begin{lemma}\label{lem:quadrule}
  Let $E$ be an edge of a triangle $K$ and $\V$ its opposite vertex.
  Then, for each polynomial $q\in P^3(K)$, we have
  \begin{equation}\label{eq:zEVqK}
    \int_K \st_{E\V}\ q \ dA = \frac{|K|}{100} q(\V).
  \end{equation}
\end{lemma}
\begin{proof}
  In the  reference triangle $\widehat{K}$  with its vertices $(0,0), (1,0), (0,1)$, let
   $E=\{(x,0)\hspace{1pt}:\hspace{1pt}0\le x \le 1\}$ with  its opposite vertex $\V=(0,1)$.
  By an affine map $\widehat{K} \rightarrow K$,
  sting functions on $K$
  are pulled back to $\st_{E\V}$ on $\widehat{K}$. 
 Thus,  it is sufficient to prove \eqref{eq:zEVqK} for
  $\st_{E\V}$ and a cubic polynomial $q$ in $\widehat{K}$.
  
  By definition in \eqref{def:zEV}, we have
  \begin{equation}\label{eq:zEVKhat}
    \st_{E\V}(x,y) = \frac1{10}(56y^3-63y^2+18y-1).
  \end{equation}
 The graph of $\st_{E\V}$ is depicted in Figure \ref{fig:spu_ref}. 

  By simple calculation, we have
  \begin{equation}\label{eq:56t3_1}
  \int_0^1 (56s^3-105s^2+60s-10)s^{k}\ ds =
    \left\{\arraycolsep=1.4pt\def\arraystretch{1.6} \begin{array}{cl}
             
              \disp \frac{-1}{20}, & \mbox{ if } k=1,\\
                                  0, & \mbox{ if } k=2,3,4.
            \end{array}\right.
        \end{equation}
        We also note
        \begin{equation}\label{eq:56t3_2}
      56t^3-63t^2+18t-1 = -56(1-t)^3+105(1-t)^2 -60(1-t)+10.
    \end{equation}
  
  Let $q=(1-y)^mx^n$ be a polynomial for
  nonnegative integers $m, n$ such that $m+n \le 3$.
  From \eqref{eq:zEVKhat}-\eqref{eq:56t3_2}, we can expand that
  \begin{equation*}\arraycolsep=1.4pt\def\arraystretch{2.4}
     \begin{array}{l}
   \disp \int_{\widehat{K}} \st_{E\V}(x,y)q(x,y)\ dA
      =\frac1{10}\int_0^1 (56y^3-63y^2+18y-1)(1-y)^m \int_0^{1-y} x^n\ dxdy \\
      \qquad
      =\disp\frac{-1}{10(n+1)}\int_0^1
      \big(56(1-y)^3-105(1-y)^2+60(1-y)-10\big)(1-y)^{m+n+1}\ dy \\
    
       \qquad
       =\disp\frac{-1}{10(n+1)}
       \int_0^1\big(56s^3-105s^2+60s-10\big) s^{m+n+1}\ ds
       =\disp\frac1{200} q(0,1)
       = \disp\frac{|\widehat{K}|}{100} q(\V).
    \end{array}
  \end{equation*}
\end{proof}
\begin{figure}[t]
\hspace{4.5cm}
\includegraphics[width=0.40\textwidth]{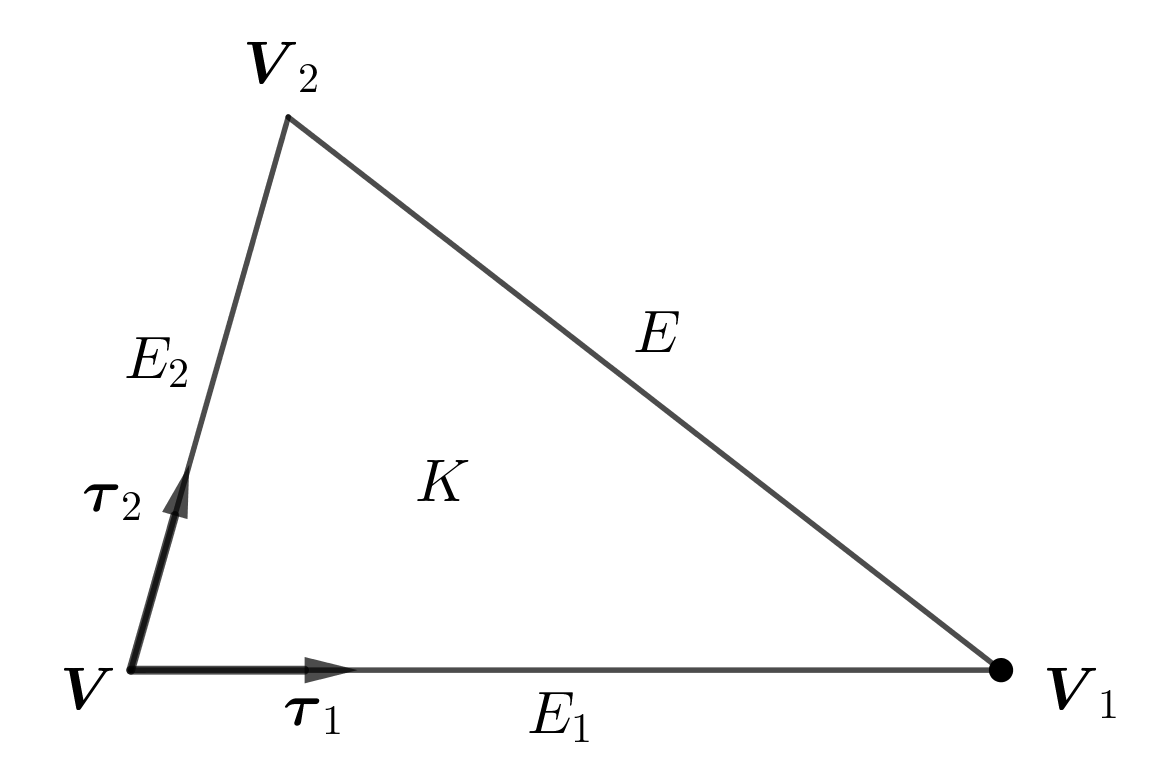}
\caption{Counterclockwisely numbered unit vectors $\t_1, \t_2$
  directed to other vertices from $\V$}
\label{fig:naming2}
\end{figure}
\begin{lemma}\label{lem:quadrature2}
  Let $E$ be an edge of a triangle $K$ and $\V$ its opposite vertex.
  Denote by $\t_1, \t_2$,   the counterclockwisely numbered unit vectors
  directed to other vertices $\V_1, \V_2$ from $\V$ 
  as in Figure \ref{fig:naming2}.
  Then for all $\v_h\in X_h$, we have
 \begin{equation*}\label{eq:zEVdivvh}
  (\st_{E\V}, \div \v_h)_K =\frac{|E_1| |E_2|}{200}
  \Big(\frac{\partial \v_h}{\partial \t_2}(\V) \cdot \t_1^{\perp}
  -  \frac{\partial \v_h}{\partial \t_1} (\V)\cdot \t_2^{\perp} \Big),
\end{equation*}
where $E_1, E_2$ are the edges sharing $\V$ 
and $\t_i^{\perp}$ is the 90-degree counterclockwise rotation of $\t_i, i=1,2$.

\end{lemma}
\begin{proof}
  For  $\v_h=(v_1, v_2)$, we write $\disp\frac{\partial \v_h}{\partial \t_1},
  \frac{\partial \v_h}{\partial \t_2} $ at $\V$ in the matrix form:  
  \begin{equation*}
    \left(\arraycolsep=1.4pt\def\arraystretch{1.5}
      \begin{array}{c} \nabla v_1(\V)^{t} \\ \nabla v_2(\V)^{t}
           \end{array}\right)
         (\t_1\ \t_2) = \Big(\frac{\partial \v_h}{\partial \t_1}(\V)\ \frac{\partial \v_h}{\partial \t_2}(\V)\Big),
       \end{equation*}
       where all vectors are presented in column forms.
       Then we expand that
       \begin{equation}\label{eq:spandivvh}
         \arraycolsep=1.4pt\def\arraystretch{1.5}
         \begin{array}{ll}
         \div \v_h (\V) &=
         \mbox{trace} \left( \begin{array}{c} \nabla v_1(\V)^{t} \\ \nabla v_2(\V)^{t}
                             \end{array}\right)=
                           \mbox{trace}\Big( (\t_1\ \t_2)^{-1}  \Big(\disp\frac{\partial \v_h}{\partial \t_1}(\V)\ \frac{\partial \v_h}{\partial \t_2}(\V)\Big) \Big) \\
                        &=\disp \frac1{\sin\theta}
                            \mbox{trace}\left( \left( \begin{array}{c} -(\t_2^{\perp})^t \\ (\t_1^{\perp})^{t}
                             \end{array}\right)
         \Big(\frac{\partial \v_h}{\partial \t_1}(\V)\ \frac{\partial \v_h}{\partial \t_2}(\V)\Big) \right)\\
         &=\disp\frac1{\sin\theta}  \Big(\t_1^{\perp}\cdot\frac{\partial \v_h}{\partial \t_2}(\V)
  - \t_2^{\perp}\cdot \frac{\partial \v_h}{\partial \t_1} (\V) \Big),             
      \end{array}
    \end{equation}
    where $\theta$ is the angle between $\t_1$ and $\t_2$.
    Since $|K|=\frac12|E_1| |E_2|\sin\theta$, we obtain \eqref{eq:zEVdivvh} with the aid of \eqref{eq:spandivvh} and
    Lemma \ref{lem:quadrule}.  
\end{proof}

\subsection{Spurious pressure}\label{sec:example_spu}
If $\Th$ has an exact singular vertex,
a spurious pressure in $M_h^S$ defined in \eqref{def:Mhs} appears.
For a simple example, let $\V$ be a boundary singular vertex which meets only one triangle $K$ in $\Th$
and has its opposite edge $E$
as $\V_4$ in Figure \ref{fig:extsingvtx}. Then, by Lemma \ref{lem:quadrule},
we obtain
\[ (\st_{E\V}, \div\v_h)_K = \frac{|K|}{100} \div\v_h (\V) =0\quad \mbox{ for all }
\v_h\in X_h,\]
since $\nabla\v_h$ vanishes at $\V$. Thus,
$\st_{E\V}-c$ is a spurious pressure in $M_h^S$
for a constant function $c$ on $\O$ such that
 $\st_{E\V}-c\in L_0^2(\O)$.

For an another example, let $\V$ be an interior singular vertex
which meets with 4 triangles $K_1, K_2, K_3, K_4$
counterclockwisely numbered as in figure \ref{fig:int_spu}.
The vertex $\V$ has 4 opposite edges $E_i\subset K_i, i=1,2,3,4$.
Denote by $\t_1, \t_2$, the counterclockwisely numbered
unit vectors at $\V$ directed other vertices in $K_1$ and
by $\ell_1, \ell_2, \ell_3, \ell_4$, the lengths of edges
corresponding to $\t_1, \t_2, -\t_1, -\t_2$, respectively.

Now, we calculate the followings by Lemma \ref{lem:quadrature2}:
\begin{equation}\label{expan:sE1234v}
  \arraycolsep=1.4pt\def\arraystretch{1.8}
  \begin{array}{lcl}
    (\st_{E_1\V}, \div\v_h)_{K_1} &=& \disp\frac{\ell_1\ell_2}{200}
  \Big(\frac{\partial \v_h}{\partial \t_2}(\V)\t_1^{\perp} -
                                      \frac{\partial \v_h}{\partial \t_1}(\V)\t_2^{\perp} \Big), \\
   (\st_{E_2\V}, \div\v_h)_{K_2} &=& \disp\frac{\ell_2\ell_3}{200}
  \Big(\frac{\partial \v_h}{\partial (-\t_1)}(\V)\t_2^{\perp} -
                                     \frac{\partial \v_h}{\partial \t_2}(\V)(-\t_1)^{\perp} \Big), \\
     (\st_{E_3\V}, \div\v_h)_{K_3} &=& \disp\frac{\ell_3\ell_4}{200}
  \Big(\frac{\partial \v_h}{\partial (-\t_2)}(\V)(-\t_1)^{\perp} -
                                       \frac{\partial \v_h}{\partial (-\t_1)}(\V)(-\t_2)^{\perp} \Big), \\
     (\st_{E_4\V}, \div\v_h)_{K_4} &=& \disp\frac{\ell_4\ell_1}{200}
  \Big(\frac{\partial \v_h}{\partial \t_1}(\V)(-\t_2)^{\perp} -
                                      \frac{\partial \v_h}{\partial (-\t_2)}(\V)\t_1^{\perp} \Big).
  \end{array}
\end{equation}
Let $q_h\in M_h$ be an alternating sum of $ \st_{E_i\V}, i=1,2,3,4$ such that
\begin{equation}\label{eq:spuqh}
  q_h = \frac1{\ell_1\ell_2} \st_{E_1\V} -  \frac1{\ell_2\ell_3} \st_{E_2\V}
+ \frac1{\ell_3\ell_4} \st_{E_3\V} -  \frac1{\ell_4\ell_1} \st_{E_4\V}. 
\end{equation}   
Then, since $\v_h$ is continuous on edges,  we have from \eqref{expan:sE1234v},
\[  (q_h, \div\v_h) =0 \quad \mbox{ for all }
\v_h\in X_h. \]
\begin{figure}[ht]
\hspace{4.5cm}
\includegraphics[width=0.45\textwidth]{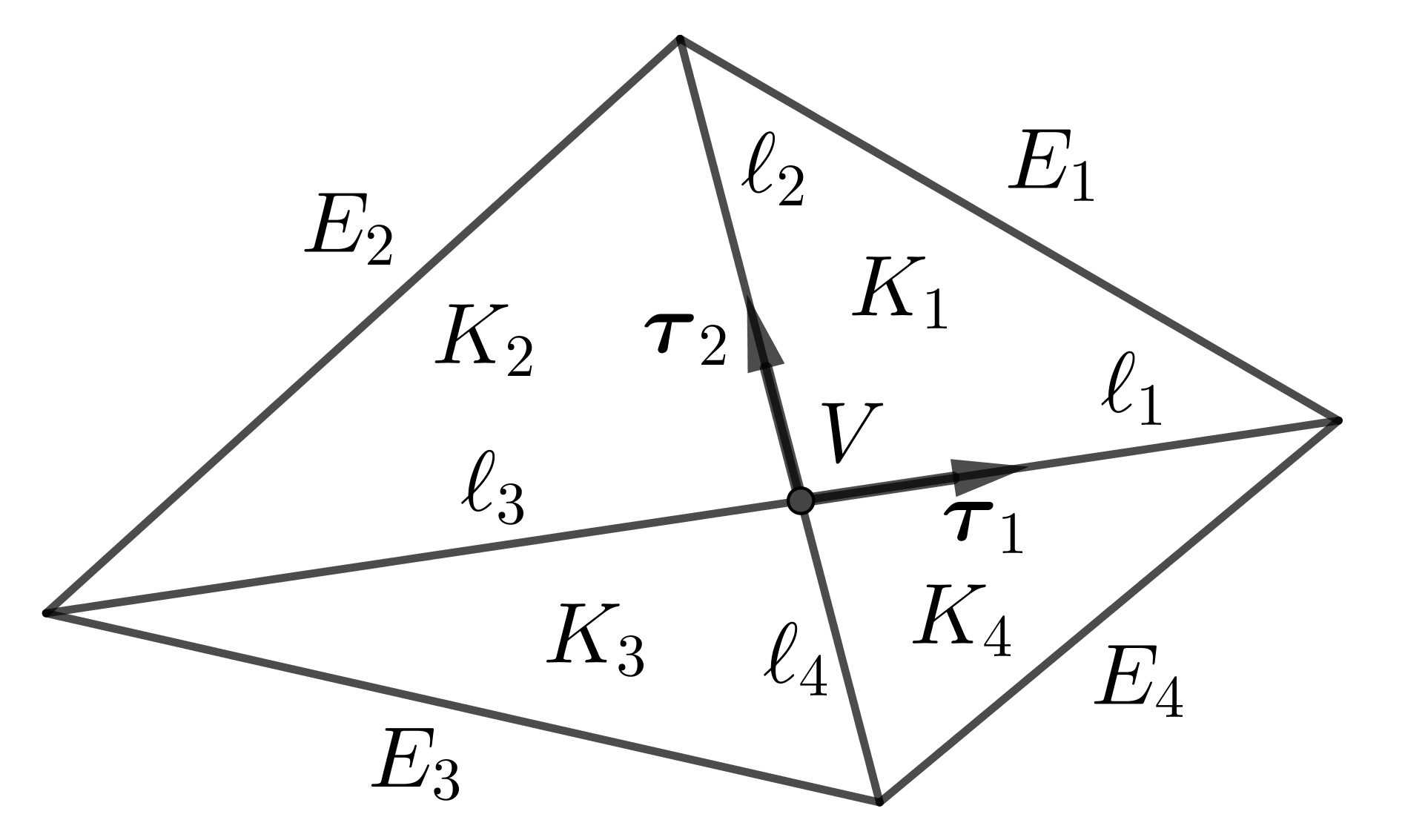}
\caption{Interior exact singular vertex $\V$ causing a spurious pressure}
\label{fig:int_spu}
\end{figure}

\section{A basis of $P^3$ over $K$}\label{sec:basis}
We will suggest a new basis of $P^3$ over a triangle $K$
which includes sting functions $\st_{E\V}$.
\subsection{16-point Lyness quadrature rule}
The following 16-point Lyness quadrature rule \cite{Lyness1975}
is exact over a triangle $K$ for any polynomial $p$ of degree up to 6:
\begin{equation}\label{eq:Lyn16}
\int_{{K}} p(x,y)\ dxdy = |K| \sum_{i=1}^{16} p(\x_i) w_i.
\end{equation}

The 16 quadrature points in \eqref{eq:Lyn16}
include the gravity center  $\G$ of $K$
and the center $\G_i$ of the segment
connecting the vertex $\V_i$ and the midpoint $\M_i$ of the opposite edge of $\V_i$,
$i=1,2,3$  as in Figure  \ref{fig:Lyn16}.
The other 12 points lie on the boundary of $K$.

\begin{figure}[ht]
\hspace{3.7cm}
\includegraphics[width=0.50\textwidth]{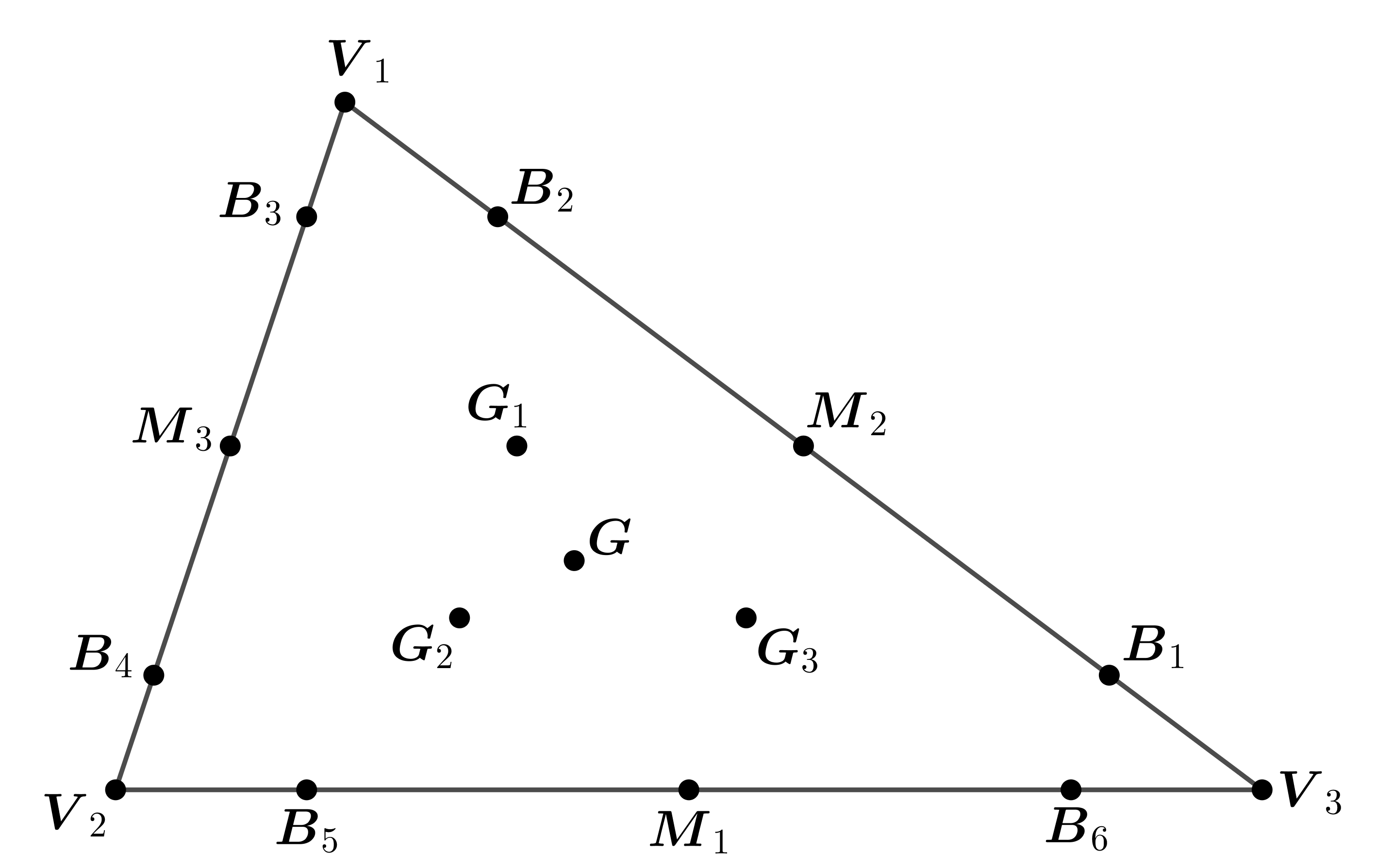}
\caption{ 16 Lyness quadrature points,
  $\G_i$ is the center of $\overline{\V_i\M_i}
   \ i=1,2,3$}\label{fig:Lyn16}
\end{figure}
In the reference triangle $\hat{K}$ with vertices $(0,0),(1,0),(0,1)$,
the 16 quadrature points and their corresponding weights are listed:
\begin{equation}\label{eq:Lynlist}
  \arraycolsep=1.4pt\def\arraystretch{1.6}
\begin{array}{ccl}
  \{\x_i\}_{1}^3&=&\{ (0,0), (1,0), (0,1)\}, \quad \{ w_i\}_1^3 = \{-5/252\},\\
 \{ \x_i\}_4^9  &=&\{(0,a), (0,b), (a,0), (b,0), (a,b), (b,a)\}, \quad \{ w_i\}_4^9 = \{3/70\}, \\
 \{ \x_i\}_{10}^{12}&=& \{ (0,1/2), (1/2,0), (1/2, 1/2)\},\quad  \{ w_i\}_{10}^{12} =  \{17/315\}, \\
  \{ \x_i\}_{13}^{15}&=&\{ (1/4, 1/4), (1/4, 1/2), (1/2, 1/4)\},\quad \{ w_i\}_{13}^{15} =\{128/315\}, \\
  \x_{16}&=& (1/3, 1/3),\quad w_{16} =-81/140,
\end{array}
\end{equation}
where $a=(3-\sqrt{6})/6,\ b=1-a$.

\subsection{Basis functions with interior Lyness points}
Let $\V$ be a vertex of a triangle $K$ and $\G$  the gravity center of $K$.
Denote by $\ii_{\V}$,  the unit vector from $\V$ to $\G$ as in Figure \ref{fig:svze}-(a),
that is 
\[ \ii_{\V} = \overrightarrow{\V\G} \Big/  |\overrightarrow{\V\G}|,   \]
and by ${\ii_{\V}}^{\perp}$, the  90-degree counterclockwise rotation of $\ii_{\V}$,
and by $\mu$,
a linear function which vanishes at the line passing $\V,\G$ such that
\begin{equation}\label{def:mu}  \mu(\x) = {\ii_{\V}}^{\perp}\cdot(\x-\G),
\end{equation}
lastly, by $d$,  the common distance from two other vertices of $K$
to the line $\mu(\x)=0$ as in Figure \ref{fig:svze}-(a).

\begin{figure}[ht]
\hspace{10mm}
\subfloat[$\mu(\x)={\ii_{\V}}^{\perp}\cdot(\x-\G),\
\ib_{\V}^{\pm}=\mu^{\pm}\Big(\disp\frac{\mu}d\Big)$]{
\includegraphics[width=0.38\textwidth]{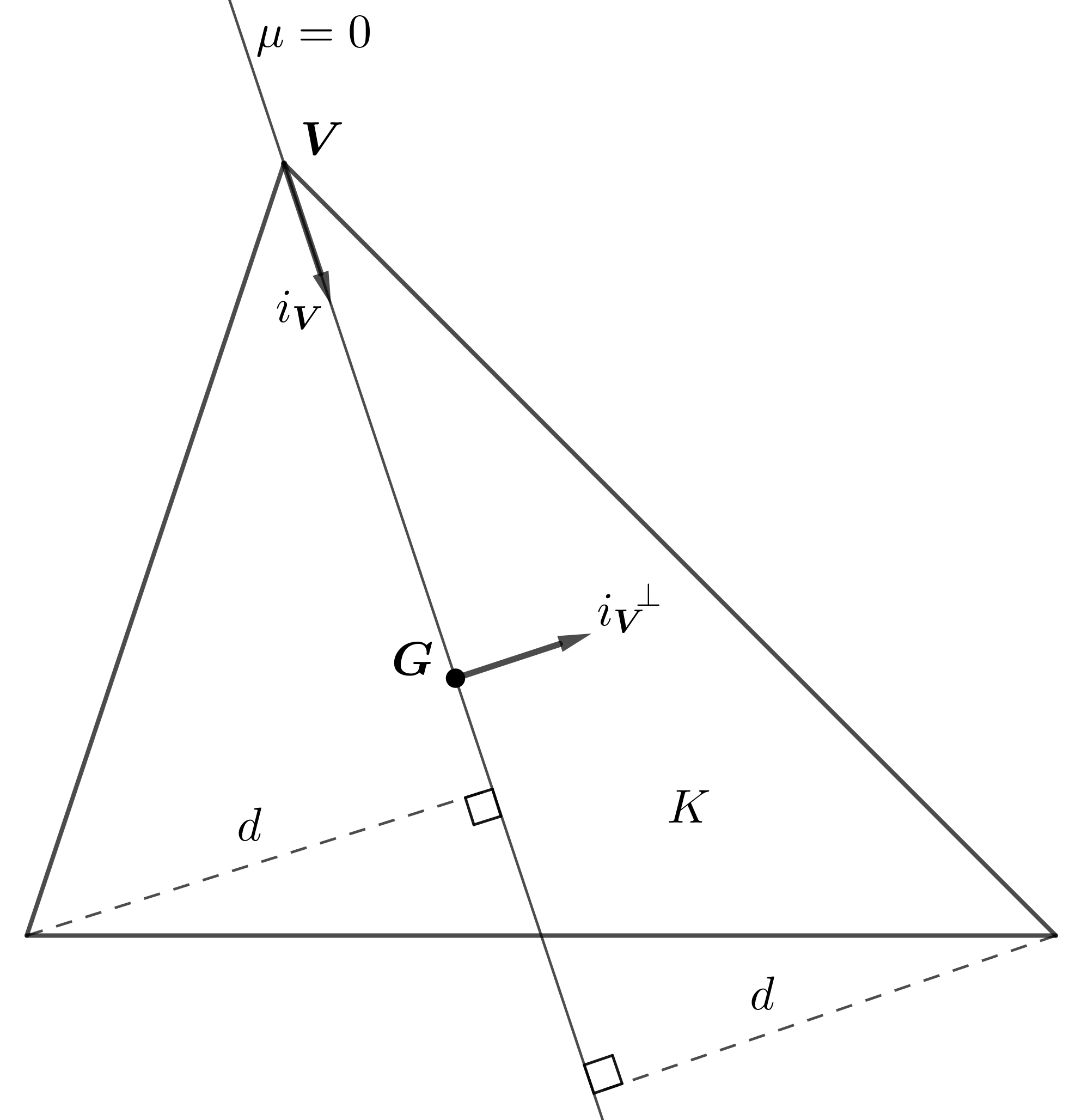}
}\qquad
\subfloat[$\lam(\x)={-\n}\cdot(\x-\M),\
\st_{E\V}=\f\Big(\disp\frac{\lam}H\Big)$ ]{
\includegraphics[width=0.38\textwidth]{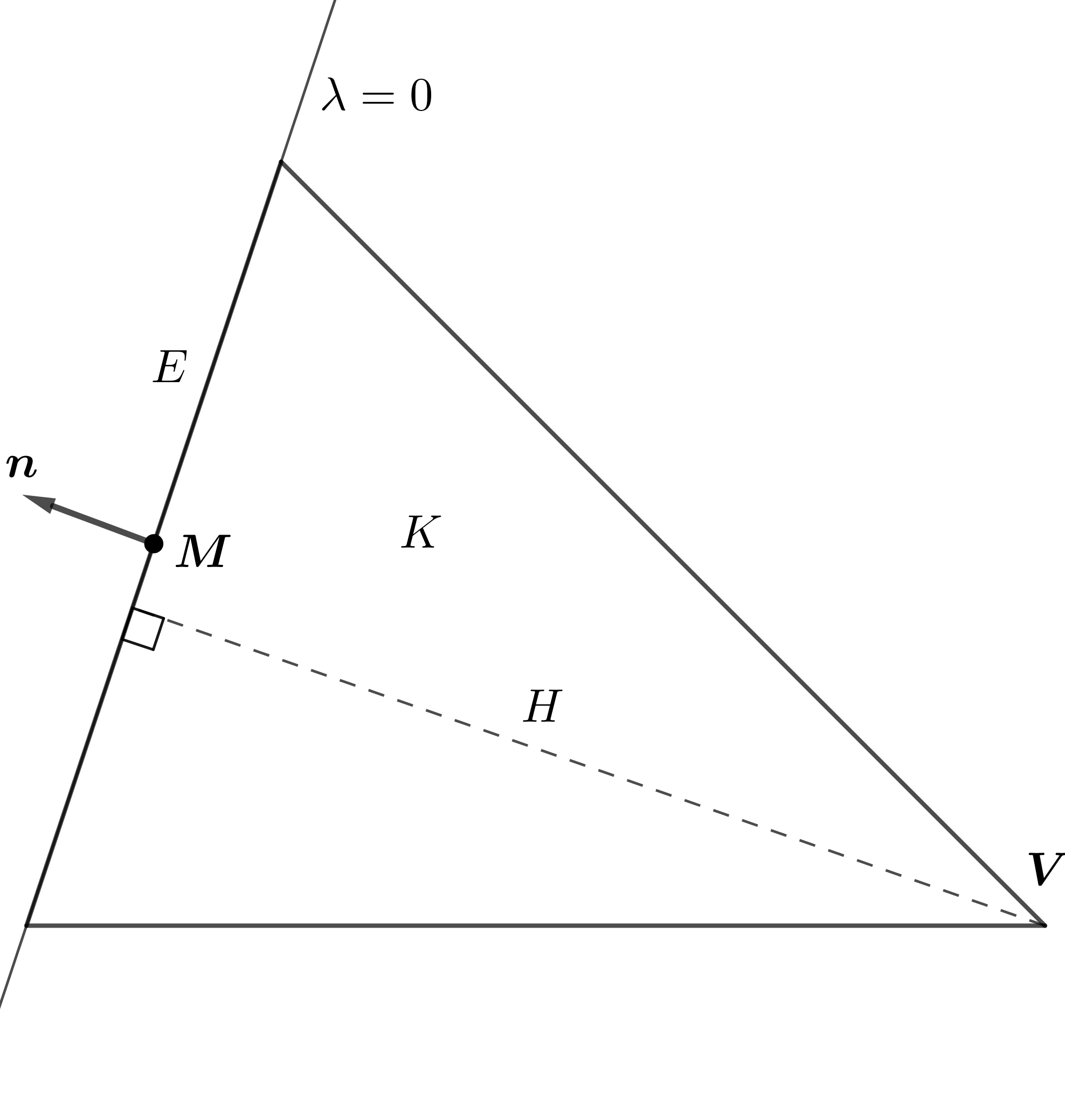}
}\caption{Definition of three basis cubic polynomials over $K$:
  $\ib_{\V}^{+},\ \ib_{\V}^{-},\ \st_{E\V}$}
\label{fig:svze}
\end{figure}
Define two basis cubic polynomial $\ib_{\V}^+, \ib_{\V}^- \in P^3(K)$
determined by $\V, \G$:
\begin{equation}\label{def:ib}
  \ib_{\V}^+(\x) = \iota^+\Big(\frac{\mu(\x)}{d}\Big),\quad
  \ib_{\V}^-(\x) = \iota^-\Big(\frac{\mu(\x)}{d}\Big),
\end{equation}
with two  auxiliary cubic functions $\iota^+, \iota^-$:
\begin{equation*}\label{def:iota}
  \iota^+(t)=8t^3+3t^2,\quad \iota^-(t)=8t^3-3t^2.  
\end{equation*}

We have chosen $\iota^{\pm}$ so that $\nabla \ib^\pm_{\V}$ vanishes at
3 points among 4 interior Lyness points $\G, \G_1, \G_2, \G_3$ of $K$
as in the following lemma.
\begin{lemma}\label{lem:gradgv}
  Let $\V$ be a vertex of a triangle $K$ and $\P$
  be among four 16-Lyness quadrature
points inside $K$.
  Then, we have
  \begin{equation}\label{eqlem:nablaib}
    \nabla\ib_\V^+ (\P) = \left\{\ 
      \arraycolsep=1.4pt\def\arraystretch{1.4}
        \begin{array}{cl}
          \disp\frac3d {\ii_{\V}}^{\perp} \quad & \mbox{ if } \mu (\P) >0,\\
          0 \quad & \mbox { otherwise },
        \end{array}\right.
      \qquad
      \nabla\ib_\V^- (\P) = \left\{\ 
        \arraycolsep=1.4pt\def\arraystretch{1.4}
        \begin{array}{cl}
           \disp\frac3d {\ii_{\V}}^{\perp}  \quad & \mbox{ if } \mu (\P) < 0,\\
          0 \quad & \mbox { otherwise }.
        \end{array}\right.
      \end{equation}
    \end{lemma}

\begin{proof}
  Let $\V^+, \V^-$ be two vertices of triangle $K$ other than $\V$ such that
  \[ \mu(\V^+)  >0, \quad  \mu(\V^-) <0.  \]
The four 16-Lyness quadrature
points inside $K$ are the gravity center $\G$ and
\begin{eqnarray*}
  \G_0=\frac12\V + \frac14\V^+ + \frac14\V^-,\quad
  \G^+=\frac14\V + \frac12\V^+ + \frac14\V^-,\quad
  \G^-=\frac14\V + \frac14\V^+ + \frac12\V^-.
\end{eqnarray*}
The two points $\G, \G_0$ lie on the line $l=\{\x : \mu(\x)=0 \}$
and we simply calculate the common distance between $l$ and $\G^\pm$
is a quarter of $d$ between $l$ and $\V^\pm$. Thus we have
\begin{equation}\label{eq:valmuv}
  \mu(\G)=\mu(\G_0) =0,\quad \mu(\G^+) =\frac d4, \quad \mu(\G^-) = -\frac d4.
\end{equation}
From the definition of $\mu, \ib_\V^+$ in \eqref{def:mu},\eqref{def:ib}, we have
\begin{equation}\label{eq:nablaib}
  \nabla\ib_\V^+(\x) = \frac1d {\iota^+}'\Big(\frac{\mu(\x)}{d}\Big) {\ii_\V}^{\perp}.  
\end{equation}
We prove \eqref{eqlem:nablaib} for $\nabla\ib_\V^+ $
by \eqref{eq:valmuv}, \eqref{eq:nablaib}, since
${\iota^+}'(0)={\iota^+}'(-1/4)=0,\ {\iota^+}'(1/4)=3$.
We can repeat the same argument for $\nabla\ib_\V^-$ in \eqref{eqlem:nablaib}.
\end{proof}

Now, we form a new basis of $P^3$ over $K$ in the following lemma.

\begin{lemma}\label{lem:P3basis}
  Let $K$ be a triangle with
   vertices $\V_1, \V_2, \V_3$ and their respective opposite edges $E_1, E_2, E_3$.
  Then, we have
  \begin{equation}\label{eq:spanP3}
    P^3\ =\ < 1, \ib_{\V_1}^+, \ib_{\V_1}^-,  \ib_{\V_2}^+,  \ib_{\V_2}^-,  \ib_{\V_3}^+,
    \ib_{\V_3}^-,\st_{E_1\V1}, \st_{E_2\V2},\st_{E_3\V3} >.
  \end{equation}
\end{lemma}
\begin{proof}
 Assume a linear combination $q$ of 10 functions in \eqref{eq:spanP3} vanishes, that is,
  \begin{equation*}
    q=c_1+ c_2\ib_{\V_1}^+ +c_3 \ib_{\V_1}^- + c_4  \ib_{\V_2}^+ + c_5  \ib_{\V_2}^-
    +c_6\ib_{\V_3}^+ + c_7\ib_{\V_3}^-+ c_8 \st_{E_1\V_1} + c_9 \st_{E_2\V_2} 
+ c_{10} \st_{E_3\V_3} =0, 
  \end{equation*}
  for some scalars $c_1, c_2,\cdots,c_{10}$.
  
  As in Figure \ref{fig:Lyn16}, let $\G_i$ be the interior 16-Lyness points
  corresponding to $\V_i$, $i=1,2,3$. We can choose  a quartic polynomial $v$ vanishing on $\p K$ and satisfying  \[v(\G_1)=v(\G_2)=0,\quad v(\G_3)=1.\]
For two scalars $\alpha, \beta$,  define
\[ \v=(\alpha , \beta )v.\]
We note from the quadrature rule in Lemma \ref{lem:quadrule}, 
\begin{equation}\label{eq:sevdivani}
  (\st_{E_i\V_i}, \div \v)=0,\quad i=1,2,3.
\end{equation}
Thus, by 16-Lyness quadrature rule in \eqref{eq:Lyn16}, \eqref{eq:Lynlist}
and the property of $\nabla\ib_{\V_i}^{\pm}, i=1,2,3$ 
in Lemma \ref{lem:gradgv}, we expand
\begin{equation}\label{expan:basisgv}
  \arraycolsep=1.4pt\def\arraystretch{1.6}
    \begin{array}{lcl}
    0&=&(q, \div\v)=-(\nabla q, \v) 
= -\nabla (c_2\ib_{\V_1}^+  + c_5\ib_{\V_2}^-)(\G_3)\cdot\v(\G_3)w_{15}|K|\\
&=& (c_2 \gam_1{\ii_{\V_1}}^{\perp} + c_5 \gam_2 {\ii_{\V_2}}^{\perp}) \cdot (\alpha,\beta)w_{15}|K|,  
  \end{array}
\end{equation}
for some nonzero scalars $\gam_1, \gam_2$.

If we choose $(\alpha,\beta)={\ii_{\V_2}}$ in \eqref{expan:basisgv},
we conclude $c_2=0$ and sequentially $c_5=0$, 
since  $\ii_{\V_1}, \ii_{\V_2}$ are not parallel.
By similar argument, we have $c_3=c_4=c_6=c_7=0$.

Now choose a cubic polynomial $p$  such  that its  mean over $K$ vanishes and 
 \[ p(\V_1)=1,\quad p(\V_2)=p(\V_3)=0.\]
 Then, by quadrature rule in Lemma \ref{lem:quadrule}, we have
 \[ 0 = (q,p) = (c_8 \st_{E_1\V_1},p)= c_8\frac{|K|}{100} p(\V_1).\]
 Thus, $c_8=0$ and similarly, $c_9=c_{10}=0$.
 It completes the proof, since  $\dim P^3=10$.
\end{proof}

%%%%%%%%%%%%%%%%%%%%%%%%%%%%%%%%%%%%%%%%%%%%%%%%%%%%%%%%%%%%%
\section{Error in pressure}\label{sec:errprs}
    Let
    $(\u,p)\in [H_0^1(\O)]^2\times L_0^2(\O)$
    and
    $(\u_h, p_h)\in  X_h\times M_h$
    be the solutions for the continuous and discrete Stokes problems
    \eqref{prob:conti}, \eqref{prob:discr}, respectively.
    There exists a standard projection $\Pi_h p \in M_h$ of $p$
    which is continuous in $\O$. Denoting the error in pressure by
    \begin{equation}\label{def:eh} 
      e_h = p_h - \Pi_h p,
    \end{equation}
    we will analyze that $e_h$ is stable except the
    spurious component of $e_h$ caused by  quasi singular vertices.

    By Theorem \ref{thm:veloerror}, we note  that, if $\u\in [H^5(\O)]^2$ and $p\in H^4(\O)$, then
  \begin{equation}\label{eq:ehdivest1}
 (e_h, \div\v_h) \le Ch^4(\upnorm)|\v_h|_1\quad \mbox{ for all } \v\in X_h,
\end{equation}
since $e_h$ satisfies
\begin{equation}\label{eq:eqeh}
  (e_h, \div\v_h)=(\nabla \u-\nabla\u_h, \nabla\v_h) + (p-\Pi_hp,\div\v_h)
  \quad \mbox{ for all } \v\in X_h . 
\end{equation}
We will split $e_h$ into the interior error $e_h^G$ and  sting error $e_h^S$:
\begin{equation}\label{eq:spliteh}
 e_h = e_h^G + e_h^S,
\end{equation}
where
\begin{equation*}
  e_h^G\big|_K \in<1, \ib_{\V_1}^+,  \ib_{\V_1}^-,  \ib_{\V_2}^+,  \ib_{\V_2}^-,  \ib_{\V_3}^+,  \ib_{\V_3}^->,\quad e_h^S\big|_K \in < \st_{E_1\V_1},  \st_{E_2\V_2},  \st_{E_3\V_3} >,
\end{equation*}
for each $K\in \Th$ with vertices $\V_1, \V_2, \V_3$ and
their respective opposite edges $E_1, E_2, E_3$.

 For each vertex $\V$, let $\EE_{\V}$ be a set
 of all opposite edges of $\V$. 
Then, we can cluster the sting error $e_h^S$ by vertices as
\begin{equation}\label{eq:spliteh2} e_h^S = \sum_{\V\in\VV_h} e_h^{\V},  \end{equation}
where 
 \[ e_h^{\V} \in <\st_{E_1\V},\st_{E_2\V},\cdots,\st_{E_J\V}>, \]
 for all opposite edges $E_j\in \EE_{\V},\ j=1,2,\cdots,J=\#\EE_{\V}.$

 In the remaining of this section, we will show the error $e_h$ is stable
 except the sting error $e_h^{\V}$ for quasi singular vertices $\V$.

\subsection{Inequalities for $e_h^{\V}$ in back-to-back triangles}
We first estimate $\nabla e_h^G$ by choosing a proper test function
$\v_h\in X_h$ in \eqref{eq:eqeh}. 

 \begin{lemma}\label{lem:normehs}
Let $h$ be the diameter of a triangle $K$ in $\Th$. Then we have
   \[ h\|\nabla e_h^G\|_{0,K} \le \Cs (|\u-\u_h|_{1,K} + \|p-\Pi_hp\|_{0,K}).\]
 \end{lemma}
 \begin{proof}
   With the same notations in Lemma \ref{lem:P3basis}, we represent
   \[ e_h^G\big|_K =
     c_1+ c_2\ib_{\V_1}^+ +c_3 \ib_{\V_1}^- + c_4  \ib_{\V_2}^+ + c_5  \ib_{\V_2}^-
     +c_6\ib_{\V_3}^+ + c_7\ib_{\V_3}^-, \]
   for some constants $c_1, c_2,\cdots, c_7$.
   Denote by $\G_i$, the interior 16-Lyness points
   corresponding to $\V_i$, $i=1,2,3$ as in Figure  \ref{fig:Lyn16}.
   Then, there exists a unique quartic function $v\in P^4$ vanishing on $\p K$ and
   $v(\G_1)=v(\G_2)=0, v(\G_3)=1$. We note that
   \begin{equation}\label{eq:normbv}
     |v|_{1,K}\le \Cs \quad |\ib_{\V_i}^{\pm}|_{1,K} \le \Cs ,\ i=1,2,3.
   \end{equation}
   
   Choose a test function $\v_h\in X_h$ such that $\v_h|_K=v\ii_{\V_2}$
   and vanishes outside $K$. Then, we have from \eqref{eq:sevdivani} and
   Lemma \ref{lem:quadrule}, \ref{lem:gradgv}, 
   \begin{equation}\label{eq:ehdivhk1}
     (e_h,\div\v_h)= (e_h^G,\div\v_h)_K =
     (\nabla e_h^G,\v_h)_K= 3c_2{\ii_{\V_1}}^{\perp}\cdot\ii_{\V_2} w_{15} |K|/d,
   \end{equation}
   where $d$ is the distance from $\V_2$ to the line connecting
   $\V_1$ and the gravity center $\G$.

   Now, by \eqref{eq:eqeh}, \eqref{eq:normbv}, \eqref{eq:ehdivhk1},
   we estimate
   \begin{equation*}
     \|c_2\nabla\ib_{\V_1}^{+}\|_{0,K}\le \Cs h^{-1} (|\u-\u_h|_{1,K} + \|p-\Pi_hp\|_{0,K}).
   \end{equation*}
    It completes the proof,
   by repeating the same arguments for  $c_3, c_4,\cdots,c_7$.
 \end{proof}

  Let $K$ be a triangle in $\Th$ and $E$ an edge of $K$
  between two vertices $\V_1, \V_2$ of $K$.
  Denote by $\t$, the unit tangent vector of $E$, that is,
\[ \t= \overrightarrow{\V_1\V_2}\Big/|\overrightarrow{\V_1\V_2}|.\]
We need an elementary test function $v$ in the following lemma
to estimate the sting error $e_h^S$.
\begin{lemma}\label{lem:ext_test_ftn}
There exists a quartic polynomial $v\in P^4$ such that
$v$ vanishes on $\p K\setminus E$ and
\begin{equation}\label{cond:testv_0}
  \int_{E} v\ ds =0,\quad \frac{\p v}{\p\t}(\V_1)=1,\quad
  \frac{\p v}{\p\t}(\V_2)=0,
 \quad |v|_{1,K} \le \Cs |E|.
\end{equation}
\end{lemma}

\begin{proof}
  In the reference triangle $\widehat{K}$ with vertices $(0,0), (1,0), (0,1)$,
  let
  $$\widehat{E}=\{(x,0)\ :\ 0\le x \le 1\},\quad \widehat{\V}_1=(0,0),
  \quad \widehat{\V}_2=(1,0),
  \quad \widehat{\t}=(1,0).$$
  Then a quartic polynomial $\widehat{v}=x(x+y-1)^2(-5/2x+1)$
  satisfies 
  \begin{equation}\label{cond:testvhat_0}
    \int_{\widehat{E}} \widehat{v}\ ds =0,\quad
    \disp\frac{\p\widehat{v}}{\p\widehat{\t}}(\widehat{\V}_1)
    =1,\quad
    \disp\frac{\p\widehat{v}}{\p\widehat{\t}}(\widehat{\V}_2)
    =0.
  \end{equation}

  Define $v=|E|\ \hspace{1pt}\widehat{v}\circ F^{-1}$
  for an affine map $F:\widehat{K}\longrightarrow K$ such that
  $ F(\widehat{\V}_i)=\V_i, i=1,2$.
  Then, from the definition of $\widehat{v}$ and \eqref{cond:testvhat_0},
  $v$ vanishes on $\p K\setminus E$ and
  satisfies \eqref{cond:testv_0}.
\end{proof}

The sting error $e_h^{\V}$ has an interesting characteristic
for each pair of two  back-to-back triangles sharing $\V$
in the following lemma.
 \begin{lemma}\label{lem:back2back}
   Let  two triangles $K_1, K_2$ share a vertex $\V$
   and an edge $E$ as in Figure \ref{fig:back2back}.
    Assume two scalars $\alp_1, \alp_2$ make that
 \begin{equation*}
   e_h^{\V}\big|_{K_1\cup K_2} = \alp_1 \st_{E_1\V} + \alp_2 \st_{E_2\V},   
 \end{equation*}
 for two opposite edges $E_1, E_2$ of $\V$ in $K_1, K_2$, respectively.
   Then for any unit vector $\bxi$, we have
  \begin{equation}\label{eq:ciEi}
  \left|  (\alp_1 \overrightarrow{\V\V_1} - \alp_2 \overrightarrow{\V\V_2})
    \cdot \bxi \right| \le \Cs(|\u-\u_h|_{1,K_1\cup K_2} + \|p-\Pi_hp\|_{0,K_1\cup K_2}),
  \end{equation}
  where 
  $\V_1, \V_2$ are the respective opposite vertices of $E$
  in $K_1, K_2$.
\end{lemma}
 \begin{proof}
   Let $\V_0$ be the vertex of $E$ other than $\V$ and
   $\t$ unit vector such that
   \[ \t= \overrightarrow{\V\V_0}\Big/|\overrightarrow{\V\V_0}|.\]
   From  Lemma \ref{lem:ext_test_ftn}, there exists a quartic function $v_i$
   on $K_i, i=1,2$ such that $v_i$ vanishes on $\p K_i\setminus E$ and
   \begin{equation}\label{cond:viK12}
     \int_{E} v_i\ ds =0,\quad \frac{\p v_i}{\p\t}(\V)=1,\quad
  \frac{\p v_i}{\p\t}(\V_0)=0,
  \quad |v_i|_{1,K_i} \le \Cs |E|.
\end{equation}
We note $v_1$ and $v_2$ coincide on $E$, since quartic functions
have 5 degrees of freedom on $E$.

Given unit vector $\bxi$, denote by $\bxi^{\perp}$, the 90-degree counterclockwisely
rotation of $\bxi$ and choose a test function $\v_h\in X_h$ which
vanishes outside $K_1\cup K_2$ and 
\begin{equation}\label{set:vhbxi}
  \v_h|_{K_i} =v_i\bxi^{\perp},\quad i=1,2.
\end{equation}
Then, from the quadrature rule in Lemma \ref{lem:quadrule}, we have
   \begin{equation}\label{eq:ehdivalp12}
     (e_h,\div\v_h)= (\alp_1\st_{E_1\V},\div\v_h)_{K_1} +
     (\alp_2\st_{E_2\V},\div\v_h)_{K_2} + (e_h^G,\div\v_h)_{K_1\cup K_2}.
   \end{equation}

First, from \eqref{eq:eqeh} and \eqref{cond:viK12}, we obtain
   \begin{equation}\label{eq:ehdiv12hmain}
    | (e_h,\div\v_h)| \le \Cs|E|(|\u-\u_h|_{1,K_1\cup K_2} + \|p-\Pi_hp\|_{0,K_1\cup K_2}).
   \end{equation} 
Second, let $\m_i$ be the mean of $e_h^G$ over $K_i$ and $h_i$ the diameter of $K_i$, $i=1,2$.
 Then, by Lemma \ref{lem:normehs}, we estimate for $i=1,2$,
 \begin{equation}\label{eq:ehGinKi}
   \arraycolsep=1.4pt\def\arraystretch{1.6}
     \begin{array}{lcl}
     |(e_h^G,\div\v_h)_{K_i}| &=& |(e_h^G-\m_i,\div\v_h)_{K_i}|\le\|e_h^G-\m_i\|_{0,K_i} |\v_h|_{1,K_i}\le
                                    \Cs h_i |e_h^G|_{1,K_i} |\v_h|_{1,K_i} \\
     &\le&
     \Cs |E|(|\u-\u_h|_{1,K_i} + \|p-\Pi_hp\|_{0,K_i}).  
   \end{array}
 \end{equation}
 To the last, by \eqref{cond:viK12},\eqref{set:vhbxi} and  Lemma \ref{lem:quadrature2},
 we have
   \[ (\st_{E_1\V},\div\v_h)_{K_1} = \frac{|E|}{200}\bxi^{\perp}\cdot{ \overrightarrow{\V\V_1}}^{\perp},\quad
     (\st_{E_2\V},\div\v_h)_{K_2} = -\frac{|E|}{200}\bxi^{\perp}\cdot{ \overrightarrow{\V\V_2}}^{\perp}.\]
   It implies that
   \begin{equation}\label{eq:alp12bxi}
     (\alp_1\st_{E_1\V},\div\v_h)_{K_1} + (\alp_2\st_{E_2\V},\div\v_h)_{K_2}
     =\frac{|E|}{200}(\alp_1  \overrightarrow{\V\V_1}
     -\alp_2  \overrightarrow{\V\V_2})\cdot\bxi.
     \end{equation}
We combine \eqref{eq:ehdivalp12} - \eqref{eq:alp12bxi} to get \eqref{eq:ciEi}.
 \end{proof}
\begin{figure}[t]
\hspace{4.5cm}
\includegraphics[width=0.4\textwidth]{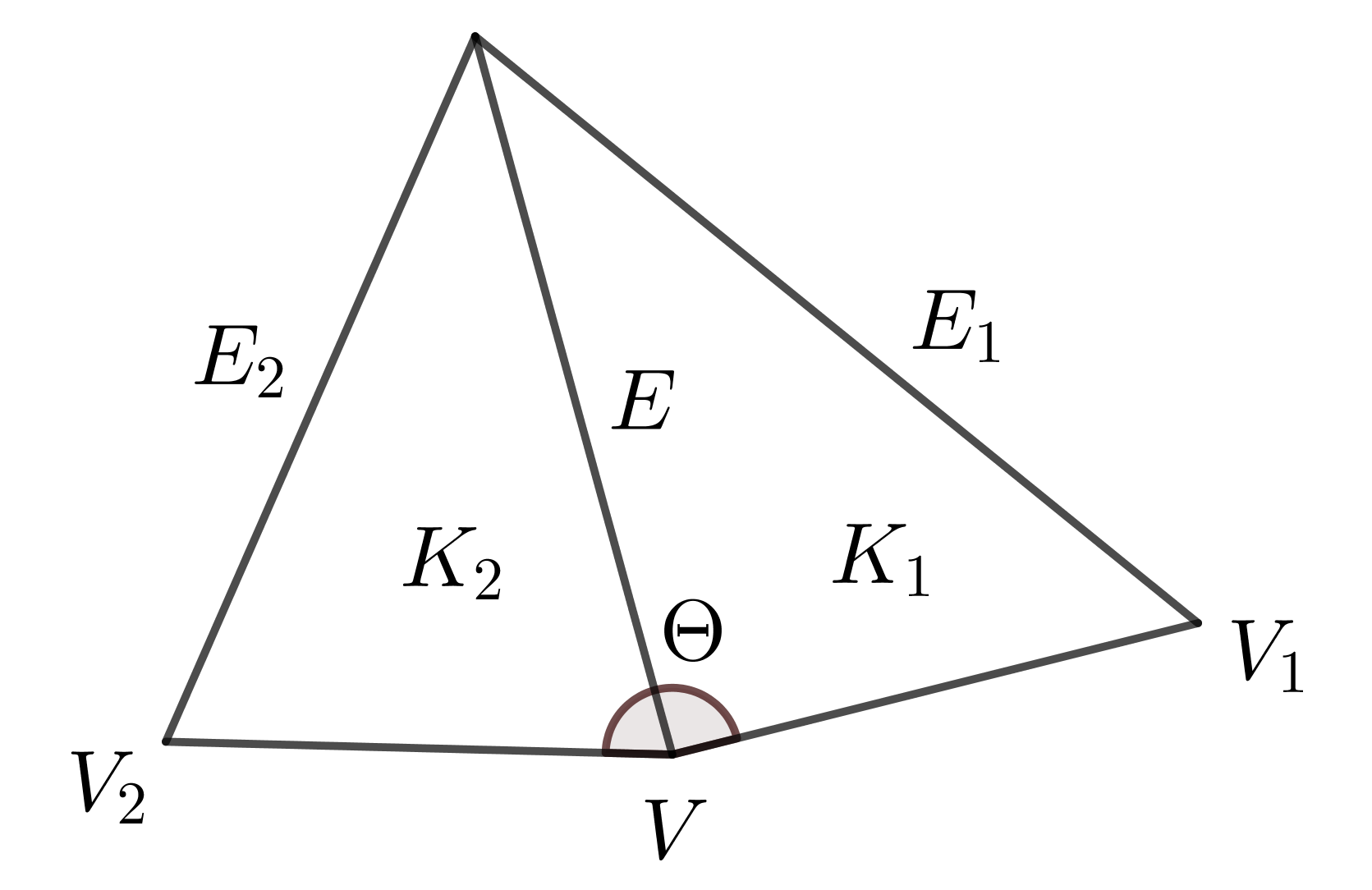}
\caption{Two back-to-back triangles $K_1, K_2$ sharing a vertex $\V$ }
\label{fig:back2back}
\end{figure}
We will choose a suitable $\bxi$ in \eqref{eq:ciEi}
to get some inequalities resulted in the following two lemmas.
They are useful in estimating the sting error $e_h^{\V}$
and postprocessing to remove the spurious error $e_h^{\V}$ for quasi singular vertices $\V$.

\begin{lemma}\label{lem:c1sin}
  Under the same assumption with Lemma \ref{lem:back2back},
 let $\Theta$ be the angle between $ \overrightarrow{\V\V_1}$
and  $ \overrightarrow{\V\V_2}$ as in Figure \ref{fig:back2back}. 
Then,
\begin{equation}\label{eq:alpsinth}
  |\alp_i \sin\Theta|\pskip |\overline{\V\V_i}| \le
     \Cs (|\u-\u_h|_{1,K_1\cup K_2} + \|p-\Pi_hp\|_{0,K_1\cup K_2}),\quad i=1,2.
   \end{equation}
 \end{lemma}   
 \begin{proof}
   Choose a unit vector $\bxi$ such that
   \begin{equation}\label{eq:choobxith}
     \Evec{\V}{\V_2}   \cdot\bxi=0.
   \end{equation}
   Then
   \begin{equation}\label{eq:cond:choobxith}
     | \Evec{\V}{\V_1}\cdot\bxi|= |\overline{\V\V_1}| \pskip |\cos(\Theta\pm\pi/2)| 
     = |\overline{\V\V_1}|\pskip |\sin\Theta|. 
   \end{equation}
   From \eqref{eq:ciEi}, \eqref{eq:choobxith}, \eqref{eq:cond:choobxith},
   we have \eqref{eq:alpsinth} for $i=1$.
   The same argument is repeated for $i=2$.
 \end{proof}

 \begin{lemma}\label{lem:c1pc2}
   Under the same assumption with Lemma \ref{lem:back2back}, we have
 \begin{equation*}
    \Big|\alp_1 |\overline{\V\V_1}| + \alp_2|\overline{\V\V_2}|\Big|
    \le \Cs(|\u-\u_h|_{1,K_1\cup K_2} + \|p-\Pi_hp\|_{0,K_1\cup K_2}).
  \end{equation*}
 \end{lemma}
 \begin{proof}
   Let  $\Theta $ be the sum of two angles of $\V$ in $K_1, K_2$
   as in Figure \ref{fig:back2back} and $0\le \dh \le \pi$ the angle
   between $\Evec{\V}{\V_1}$ and $-\Evec{\V}{\V_2}$.
  We note that, if $\Theta\le\pi$, then $\Theta+\dh=\pi$, otherwise, $\Theta-\dh=\pi$. 

  By shape regularity of $\Th$ in \eqref{def:vts}, \eqref{def:vts2},
  $\Theta$ is bounded as
  \begin{equation}\label{eq:rangeTheta}
    2\vts \le \Theta \le 2\pi -4 \vts.
  \end{equation}
  Thus, in both cases of $\Theta\le\pi$ or $\Theta>\pi$,   we have
\[ 0\le \dh \le \pi -2\vts.\]
It means
\begin{equation}\label{eq:cosdh}
  \cos(\dh/2)=\sqrt{(1+\cos\dh)/2} \ge \sqrt{(1-\cos 2\vts)/2}
  =\sin \vts > 0.
\end{equation}

Choose a unit vector $\bxi$ so that $\bxi$ forms the same acute angle $\dh/2$ with  
    $\Evec{\V}{\V_1}$ and $-\Evec{\V}{\V_2}$.
    Then, from \eqref{eq:ciEi}, \eqref{eq:cosdh}, we have
    \[ \left|\alp_1 |\overline{\V\V_1}| +  \alp_2|\overline{\V\V_2}|\right|
      \le \Cs (\sin\vts)^{-1}(|\u-\u_h|_{1,K_1\cup K_2} + \|p-\Pi_hp\|_{0,K_1\cup K_2}).\]
 \end{proof}

\subsection{Stable components and spurious error  in  $e_h$} 
For each vertex $\V$,
define the basin $\B(\V)$ of $\V$ as
the union of all triangles in $\Th$
sharing their common vertex $\V$.
For the convenience, we extend the notation as
\[ \B(\V_1, \V_2,\cdots,\V_m)=\B(\V_1)\cup \B(\V_2) \cup \cdots\cup \B(\V_m).\]

The sting error $e_h^{\V}$ has a similar property as $e_h$ in \eqref{eq:ehdivest1}
in the following lemma.
 \begin{lemma}\label{lem:ehV}
   Let $\V$ be a vertex and $\v_h \in X_h$. We have
   \begin{equation}\label{eq:ehvdivvh}
     (e_h^{\V}, \div \v_h) \le
     \Cs(|\u-\u_h|_{1,\B(\V)} + \|p-\Pi_hp\|_{0,\B(\V)})|\v_h|_{1,\B(\V)}.
   \end{equation}
 \end{lemma} 
 \begin{proof}
\begin{figure}[ht]
\hspace{3mm}
\subfloat[interior vertex $\V$]{
\includegraphics[width=0.43\textwidth]{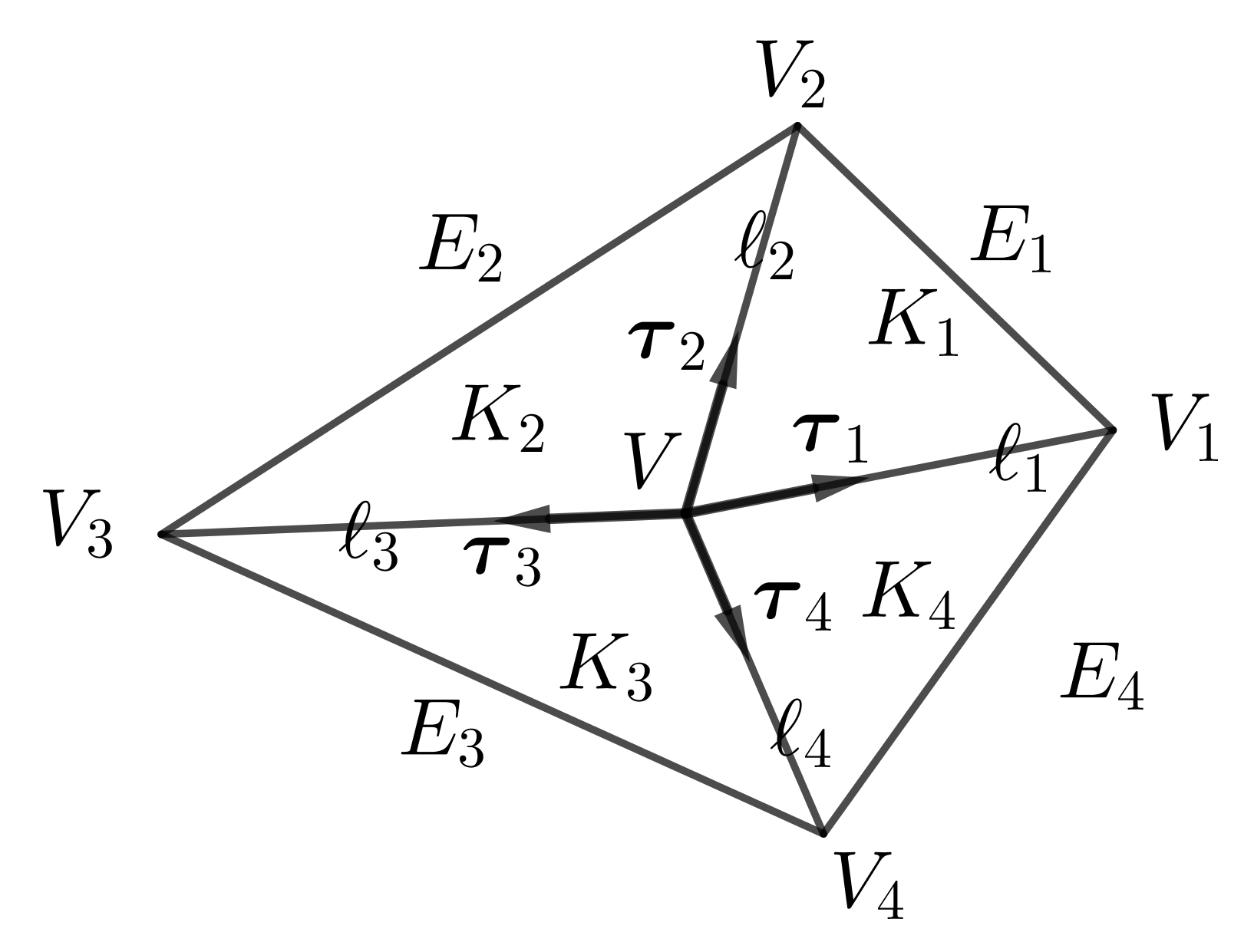}
}\qquad
\subfloat[boundary vertex $\V$]{\raisebox{5ex}
{\includegraphics[width=0.43\textwidth]{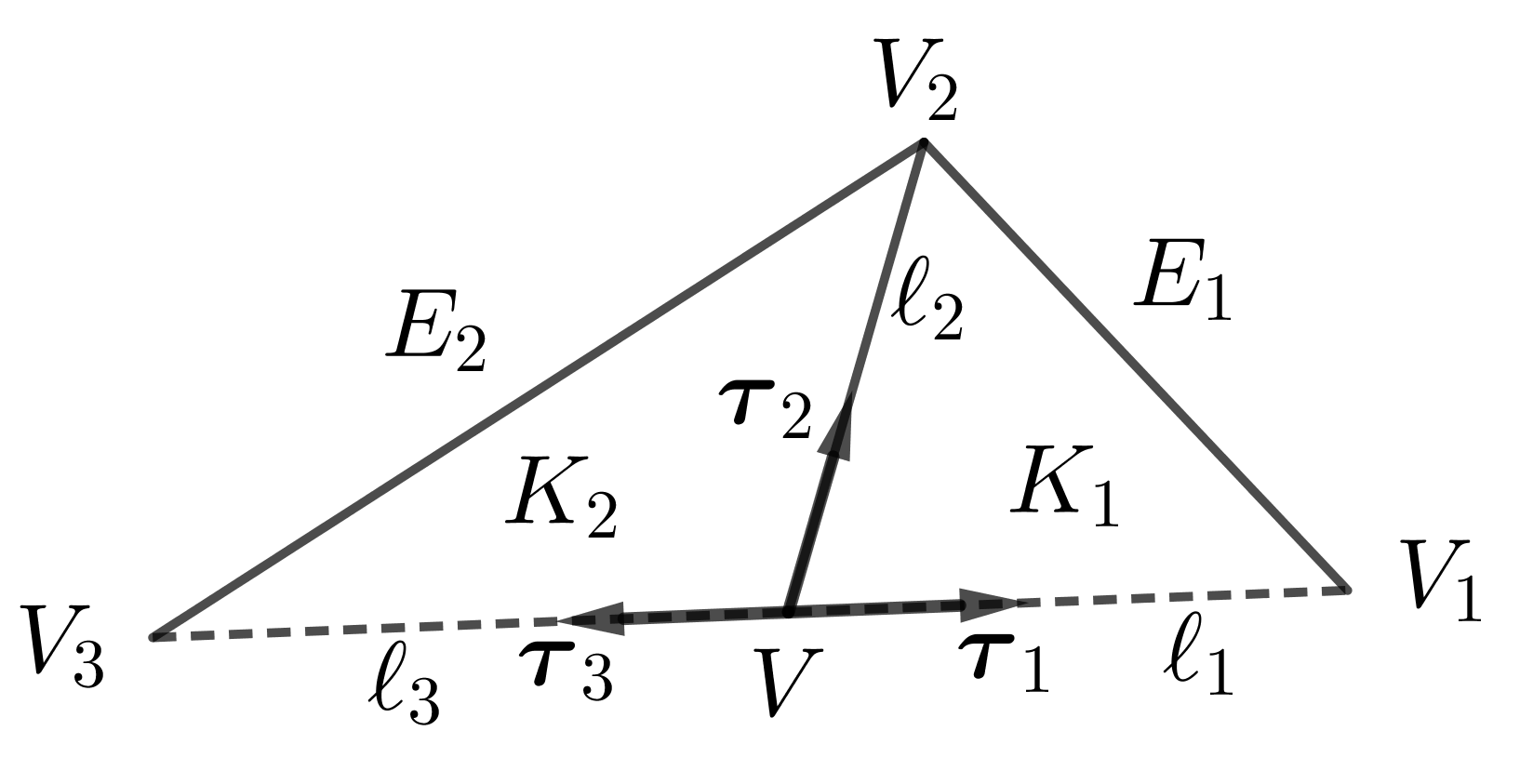} 
}}\caption{Basin $\B(\V)$ of a vertex $\V$ (dashed edges belong to $\p\O$.)}
\label{fig:basinV}
\end{figure}
Let $K_1, K_2,\cdots,K_m$ be $m$ 
triangles in $\Th$ counterclockwisely numbered
such that
\[ \B(\V)=K_1 \cup K_2\cup\cdots\cup K_m,\]
and $\V_i\in K_i, i=1,2,\cdots,m$
be consecutive 
vertices on $\p\B(\V)$ as in Figure \ref{fig:basinV}.
In case of $\V\in \p\O$, there exists one more vertex $\V_{m+1}\in K_m$ 
on $\p\B(\V)$.   If $m=1$, $\V$ belongs to $\p\O$ and 
as in subsection \ref{sec:example_spu}, 
   \[  (e_h^{\V}, \div \v_h) =0.\]

Let $m\ge2$ and $\ell_i=|\overline{\V\V_i}|$ and
$ \t_i = \Evec{\V}{\V_i}/| \Evec{\V}{\V_i}|,\ i=1,2,\cdots,m$.
Denoting by $E_i$, the opposite edge of $\V$ in $K_i, i=1,2,\cdots,m$,
 there exist
 $m$ constants $\alp_1,\alp_2,\cdots,\alp_m$ which represent 
\begin{equation}\label{rep:ehValp123}
 e_h^{\V} = \alp_1 \st_{E_1\V} +  \alp_2 \st_{E_2\V} + \cdots +  \alp_m \st_{E_m\V}.
\end{equation}
Then, from the quadrature rule in Lemma \ref{lem:quadrature2}, we have
\begin{equation}\label{eq:consehV}
  (e_h^{\V}, \div \v_h) = \sum_{i=1}^m
  \frac{\ell_i}{200}\frac{\partial \v_h}{\partial\t_i}(\V)
   \cdot\left(\alp_{i-1} {\Evec{\V}{\V_{i-1}}}^{\perp} 
     - \alp_{i} {\Evec{\V}{\V_{i+1}}}^{\perp}\right),
 \end{equation}  
where all indexes are modulo $m$, if $\V$ is an interior vertex.

We note that
 \begin{equation}\label{eq:elldvh} \ell_i\Big|\frac{\partial \v_h}{\partial\t_i}(\V)\Big| \le \Cs |\v_h|_{1,K_i},
\quad i=1,2,\cdots,m.
\end{equation}
Thus, the representation in \eqref{eq:consehV}
establishes \eqref{eq:ehvdivvh} with \eqref{eq:elldvh} and Lemma \ref{lem:back2back}.
 \end{proof}
 If a vertex $\V$ is not quasi singular, then we estimate $\nabla e_h^{\V}$
 in the following lemma.
 \begin{lemma}\label{lem:gradehV}
   Let $\V\notin \SS_h$ be a regular vertex and
   $h$ the diameter  of the basin $\B(\V)$. Then we have
 \begin{equation}\label{eq:lemnablaehV} h\| \nabla e_h^{\V}\|_{0,\B(\V)} =
     \Cs (|\u-\u_h|_{1,\B(\V)} + \|p-\Pi_hp\|_{0,\B(\V)}).   
   \end{equation}
 \end{lemma}
 \begin{proof}
Under the same notations in the proof of Lemma \ref{lem:ehV},
from the definition of $\SS_h$ in \eqref{def:sing}, 
there exist two back-to-back triangles $K_j, K_{j+1}$ such that
the sum $\Theta$ of their angles of $\V$ satisfies
\begin{equation}\label{eq:dThetapi} | \Theta-\pi | \ge \vts.    
\end{equation}

Then, from \eqref{eq:rangeTheta}, \eqref{eq:dThetapi},
we have 
$|\sin\vts|\le |\sin\Theta| $.
Thus, by  Lemma \ref{lem:c1sin},
$|\alp_j|$ in  \eqref{rep:ehValp123}  is bounded by
\[  h^{-1} \Cs (|\u-\u_h|_{1,\B(\V)} + \|p-\Pi_hp\|_{0,\B(\V)}),\]
and sequentially so are all $|\alp_i|, i=1,2,\cdots,m$ in \eqref{rep:ehValp123} 
by Lemma \ref{lem:c1pc2}.
It implies \eqref{eq:lemnablaehV}, since
\[ \|\nabla \st_{E_i\V} \|_{0,K_i} \le \Cs,\quad i=1,2,\cdots,m.\]
 \end{proof}
Split the sting error $e_h^S$ into two components by regular and quasi singular vertices: 
\begin{equation}\label{def:ehs}
  e_h^S = e_h^{SR}  + e_h^{SS},
\end{equation}
where
\[  e_h^{SR} = \sum_{\V\notin\SS_h} e_h^{\V},\quad e_h^{SS} = \sum_{\V\in\SS_h} e_h^{\V}.\]
Then the components $e_h^G ,e_h^{SR}$ in $e_h=e_h^G +e_h^{SR}+ e_h^{SS}$ is stable
as in the following theorem.
\begin{theorem}\label{thm:ehGSR}
  Let $\m$ be the mean of $e_h^G + e_h^{SR}$ over $\O$. Then,
if $\u\in[H^5(\O)]^2, p\in H^4(\O)$, we have  
  \begin{equation}\label{eq:thmehGSR}
    \|e_h^G + e_h^{SR}  - \m\|_{0} \le Ch^4 (|\u|_5 + |p|_4).
\end{equation}
\end{theorem}
\begin{proof}
  Denote $e_h^G+e_h^{SR}-\m$ by $e_h^{GR\m}$.
  Let $\Pi_he_h^{GR\m} $ be the projection of $e_h^{GR\m}\in  L_0^2(\O)$
  into $\mathcal{P}_{0,h}(\O)$. 
  Then,
  from the stability of $P^2-P^0$ \cite{Bernardi1985}, 
  there exists $\v_h\in X_h$ such that
  \begin{equation}\label{eq:ehGalp}
    ( \Pi_h e_h^{GR\m} , e_h^{GR\m} -\div\v_h) =0,\quad
    |\v_h|_1 \le C \norm{e_h^{GR\m}}{0}.
  \end{equation}

  We note, by  Theorem \ref{thm:veloerror} and Lemma \ref{lem:normehs}, \ref{lem:gradehV},
  \begin{equation}\label{eq:diffehGR}
    \| e_h^{GR\m} - \Pi_h e_h^{GR\m} \|_0 \le C h^4(|\u|_5 + |p|_4).
  \end{equation}
 Then, Lemma  \ref{lem:ehV} helps us to estimate \eqref{eq:thmehGSR}
  with  \eqref{eq:ehdivest1}, \eqref{eq:ehGalp}, \eqref{eq:diffehGR}
  in the following expansion:
\begin{equation*} \label{eq:ehGRexpan}
\arraycolsep=1.4pt\def\arraystretch{1.8}
        \begin{array}{ccl}
  \|e_h^{GR\m}\|_{0}^2 &=&
                          (e_h^{GR\m},e_h^{GR\m}  -\div\v_h) + (e_h^{GR\m}, \div\v_h) \\
          &=&(e_h^{GR\m}-\Pi_h e_h^{GR\m}, e_h^{GR\m}  -\div\v_h) + (e_h^{GR\m}, \div\v_h) \\
&\le& Ch^4(|\u|_5 + |p|_4) \|e_h^{GR\m}\|_0 + (e_h^{GR\m}, \div\v_h) \\
&=&  Ch^4(|\u|_5 + |p|_4) \|e_h^{GR\m}\|_0 
    +(e_h,\div\v_h) - (e_h^{SS},\div\v_h) \\
                     &\le& Ch^4 (|\u|_5 + |p|_4)  (\|e_h^{GR\m}\|_{0} + |\v_h|_1) 
                           \le  Ch^4 (|\u|_5 + |p|_4)  \|e_h^{GR\m}\|_{0}.
        \end{array}
      \end{equation*}
\end{proof}

If $\Th$ has no quasi singular vertex, Theorem \ref{thm:ehGSR}
asserts that $p_h -p$ has an error decay of optimal order
as expected from the inf-sup condition in \eqref{cond:infsup}.

The presence of quasi singular vertices, however, the sting error $e_h^{SS}$
could appear as large as spoiling the discrete pressure $p_h$
as in Figure \ref{fig:pandph} in the last section.
In  the next section, we are going to postprocess $p_h$ to remove $e_h^{SS}$
which is called spurious error.

\section{Remove spurious error $e_h^{SS}$}\label{sec:remove}
We will postprocess $p_h$ to remove the undesired error $e_h^{\V}$
in the following order:
\begin{enumerate}
\item $e_h^{\V}$ for interior quasi singular vertices $\V$
  using the jump of $p_h$ at $\V$,
\item $e_h^{\V}$ for boundary quasi singular vertices $\V$ away from
  corners using the jump at $\V$,
\item $e_h^{\V}$ for boundary quasi singular corners $\V$
  using the jump at the opposite edge.
\end{enumerate}

Dividing quasi singular vertices by interior and boundary into
\begin{equation*}
  \SS_h^{i}=\{\V\in\SS_h\ |\ \V \notin \p\O\},\quad
    \SS_h^{b}=\{\V\in\SS_h\ |\ \V \in \p\O\},
  \end{equation*}
  we split the spurious error $e_h^{SS}$  into
 \begin{equation}\label{eq:splitss}   e_h^{SS} =e_h^{SSi} +e_h^{SSb},
 \end{equation}
 where  
  \[  e_h^{SSi} = \sum_{\V\in\SS_h^i} e_h^{\V},
    \quad e_h^{SSb} = \sum_{\V\in\SS_h^b} e_h^{\V}.\]
\subsection{Remove interior spurious  error $e_h^{SSi}$}
Let $\V\in \SS_h^i$ be an  interior quasi singular vertex,
then the basin $\B(\V)$ of $\V$ consists of 4 triangles
$K_1,K_2,K_3,K_4$ by Lemma \ref{lem:insing4}.
In this subsection, we adopt the notations in Figure \ref{fig:basinV}-(a).
Note that  4 unknown constants $\alp_1, \alp_2,\alp_3,\alp_4$ represent $e_h^{\V}$ as
\begin{equation}\label{post:in:ehValp1234}
  e_h^{\V} = \alp_1 \st_{E_1\V} +  \alp_2 \st_{E_2\V} + \alp_3 \st_{E_3\V}
  +\alp_4 \st_{E_4\V}.
\end{equation}

By Lemma \ref{lem:c1pc2}, $\alp_1, \alp_2$ satisfy
\begin{equation}\label{post:inalp12ineq}
  |\alp_1 \ell_1 + \alp_2 \ell_3 | \le
  \Cs(|\u-\u_h|_{1,K_1\cup K_2} + \|p-\Pi_hp\|_{0,K_1\cup K_2}).
\end{equation}
Note that  $e_h^{SS}\big|_{\B(\V)} = e_h^{\V}$, since $\V$ is the only
quasi singular vertex in $\B(\V)$ by  Lemma \ref{lem:isolvtx}.
Thus, from \eqref{eq:spliteh}, \eqref{def:ehs}, \eqref{post:in:ehValp1234},
we have 
\begin{equation}\label{post:ehK1K2}
 e_h\Big|_{K_1}  
 = (e_h^G + e_h^{SR}) \Big|_{K_1}   +  \alp_1 \st_{E_1\V},\quad
 e_h\Big|_{K_2}  
 = (e_h^G + e_h^{SR}) \Big|_{K_2}   +  \alp_2 \st_{E_2\V}.
\end{equation}

Define a jump of a function $f$ at $\V$ as 
\[ [[f]]_{\V}=f|_{K_1}(\V)- f|_{K_2}(\V).\]
Then, since $\Pi_h p$ has no jump at $\V$ and $\st_{E_1\V}(\V)=\st_{E_2\V}(\V)=1$,
\eqref{post:ehK1K2} makes
\begin{equation}\label{post:pha1a2}
  [[ p_h]]_{\V} = [[e_h]]_{\V}= [[ e_h^G + e_h^{SR}]]_{\V} + \alp_1 -\alp_2.
\end{equation}
Roughly speaking,  \eqref{post:inalp12ineq} and \eqref{post:pha1a2}
help us to get $\alp_1, \alp_2$
with $[[ p_h]]_{\V}$ which we can calculate.

Choose two constants $\gam_1, \gam_2$ so that
\begin{equation}\label{post:gam13ell}
\gam_1\ell_1 + \gam_2\ell_3 =0,\quad \gam_1-\gam_2=[[p_h]]_{\V}.
\end{equation}
Then, the differences $\alp_1-\gam_1, \alp_2-\gam_2$
are estimated in the following lemma.
\begin{lemma}\label{lem:alp-s}
  Let $\m$ be the mean of $e_h^G+e_h^{SR}$ over $\O$. Then we have, for $i=1,2$,
\begin{equation}\label{eq:alp-gam}
\norm{ (\alp_i-\gam_i) \st_{E_i\V} }{0,K_i}
\le \Cs (\norm{e_h^G + e_h^{SR}-\m}{0,K_1\cup K_2} 
+ |\u-\u_h|_{1,K_1\cup K_2} + \|p-\Pi_hp\|_{0,K_1\cup K_2}).
\end{equation}
\end{lemma}
\begin{proof}
  By \eqref{post:inalp12ineq}, \eqref{post:pha1a2}, \eqref{post:gam13ell},
the differences $d_1=\alp_1-\gam_1, d_2=\alp_2-\gam_2$ satisfy
\begin{equation}\label{post:ind12ineq}
  |d_1 \ell_1 + d_2 \ell_3 | \le
  \Cs(|\u-\u_h|_{1,K_1\cup K_2} + \|p-\Pi_hp\|_{0,K_1\cup K_2}),\quad
  d_1-d_2+ [[ e_h^G + e_h^{SR}]]_{\V}=0.
\end{equation}
The equation in \eqref{post:ind12ineq} induces that
\begin{equation}\label{post:estd1d2}
  |d_1-d_2| \le \Cs(  \ell_1^{-1}\|e_h^G + e_h^{SR}-\m\|_{K_1}
  +  \ell_3^{-1}\|e_h^G + e_h^{SR}-\m\|_{K_2}),
\end{equation}
since
\[  [[ e_h^G + e_h^{SR}]]_{\V} =  (e_h^G + e_h^{SR}-\m)\Big|_{K_1}(\V)-(e_h^G + e_h^{SR}-\m)\Big|_{K_2}(\V).\]

Note that
\begin{equation}\label{post:stE1V}
  \|\st_{E_1\V}\|_{0,K_1} \le \Cs \ell_1,\quad  \|\st_{E_2\V}\|_{0,K_2} \le \Cs \ell_3. 
\end{equation}
Then, combining \eqref{post:ind12ineq}-\eqref{post:stE1V},
the estimation \eqref{eq:alp-gam} comes from the following identities:
\[ (\ell_1+\ell_3)d_1=(d_1\ell_1+d_2\ell_3)+\ell_3(d_1-d_2),\quad
    (\ell_1+\ell_3)d_2=(d_1\ell_1+d_2\ell_3)-\ell_1(d_1-d_2).\]
\end{proof}
\noindent
For another pair of two triangles $K_3, K_4$,
we can choose $\gam_3, \gam_4$ in the similar way of $\gam_1, \gam_2$. 

Now, for each  interior quasi singular vertex $\V\in\SS_h^i$,
calculate such $\gam_1,\gam_2, \gam_3, \gam_4$ and define 
\[ s_h^{\V} =\gam_1 \st_{E_1\V} + \gam_2 \st_{E_2\V}
  +\gam_3 \st_{E_3\V} + \gam_4 \st_{E_4\V}, \]
and 
  \begin{equation}\label{post:def:shi}
    s_h^i = \sum_{ \V\in\SS_h^i} s_h^{\V}.
  \end{equation}
Then, from Theorem \ref{thm:veloerror}, \ref{thm:ehGSR} and
Lemma \ref{lem:alp-s}, we establish the following lemma.
\begin{lemma}\label{lem:errorbdyi}
 If $\u\in[H^5(\O)]^2, p\in H^4(\O)$, we have
  \begin{equation}\label{eq:mainstable}
    \| e_h^{SSi}-s_h^i\|_{0} \le Ch^4 (|\u|_{5} + |p|_4).
  \end{equation}
\end{lemma}

\subsection{Remove boundary  spurious error  $e_h^{SSb}$}
We have known $p_h$ and  $s_h^i$ such that
 \begin{equation}\label{post:def:ehi2}
   p_h-s_h^i -\Pi_h p = e_h^G + e_h^{SR} + e_h^{SSi}-s_h^i + e_h^{SSb}.
 \end{equation}
 In this subsection, we will deal with the error $e_h^{SSb}$ in \eqref{post:def:ehi2}
for  boundary  quasi singular vertices.

Denote by $\mathcal{R}_h$, the set all regular vertices,
that is $\mathcal{R}_h=\VV_h\setminus \SS_h$.
Let $\p\O\setminus\mathcal{R}_h$ consist of $J$ components 
$ \sg_1, \sg_2, _3,\cdots, \sg_J$ and
define quasi singular chains  as
\[ \Q_j=\VV_h \cap \sg_j,\quad j=1,2,\cdots,J.\]
Note $\Q_1, \Q_2,\cdots,\Q_J$ are sets of consecutive boundary quasi singular  vertices
separated by regular vertices.
We will first remove spurious error for all quasi singular chains
which do not contain any corner of $\p\O$
in subsubsection \ref{subsubsection:notcorner} below.
Then we will go to the remaining quasi singular chains having a corner
in  subsubsection \ref{subsubsection:corner}.

Let $\SS_h^{br}$ be the union of all quasi singular chains not having any corner
and $\SS_h^{bc}=\SS_h^b\setminus \SS_h^{br}$.
Then, split $e_h^{SSb}$ into
\begin{equation}\label{eq:splitssb}
  e_h^{SSb} = e_h^{SSbr} + e_h^{SSbc},
\end{equation}
where
\[ e_h^{SSbr} = \sum_{ \V\in\SS_h^{br}} e_h^{\V},\quad
  e_h^{SSbc} = \sum_{ \V\in\SS_h^{bc}} e_h^{\V}.\]

In the remaining analysis,
we will use the notations in this paragraph. Let $\Sg$ be  a set
of $m+2$ consecutive vertices on a
line segment of $\p\O$ such that
\begin{equation}\label{post:def:S}
  \Sg=\{ \V_0, \V_1,\cdots, \V_m, \V_{m+1}\},
\end{equation}
as in Figure \ref{fig:singbdy}.
Assume $\V_1, \V_2,\cdots,\V_m$ are quasi singular, actually exact singular.
Then, there exists a vertex $\W$ such that,
for each $k\in\{0,1,2,\cdots,m+1\}$,
there is an edge $E_k$ which connects $\W$ and $\V_k$.
Let $K_k$ be the triangle with vertices $\V_{k-1}, \V_k, \W$ and
$ \ell_k=|\V_{k-1}\V_k|,\ k=1,2,\cdots,m+1.$
\begin{figure}[ht]
  \hspace{3cm}
\includegraphics[width=0.6\linewidth]{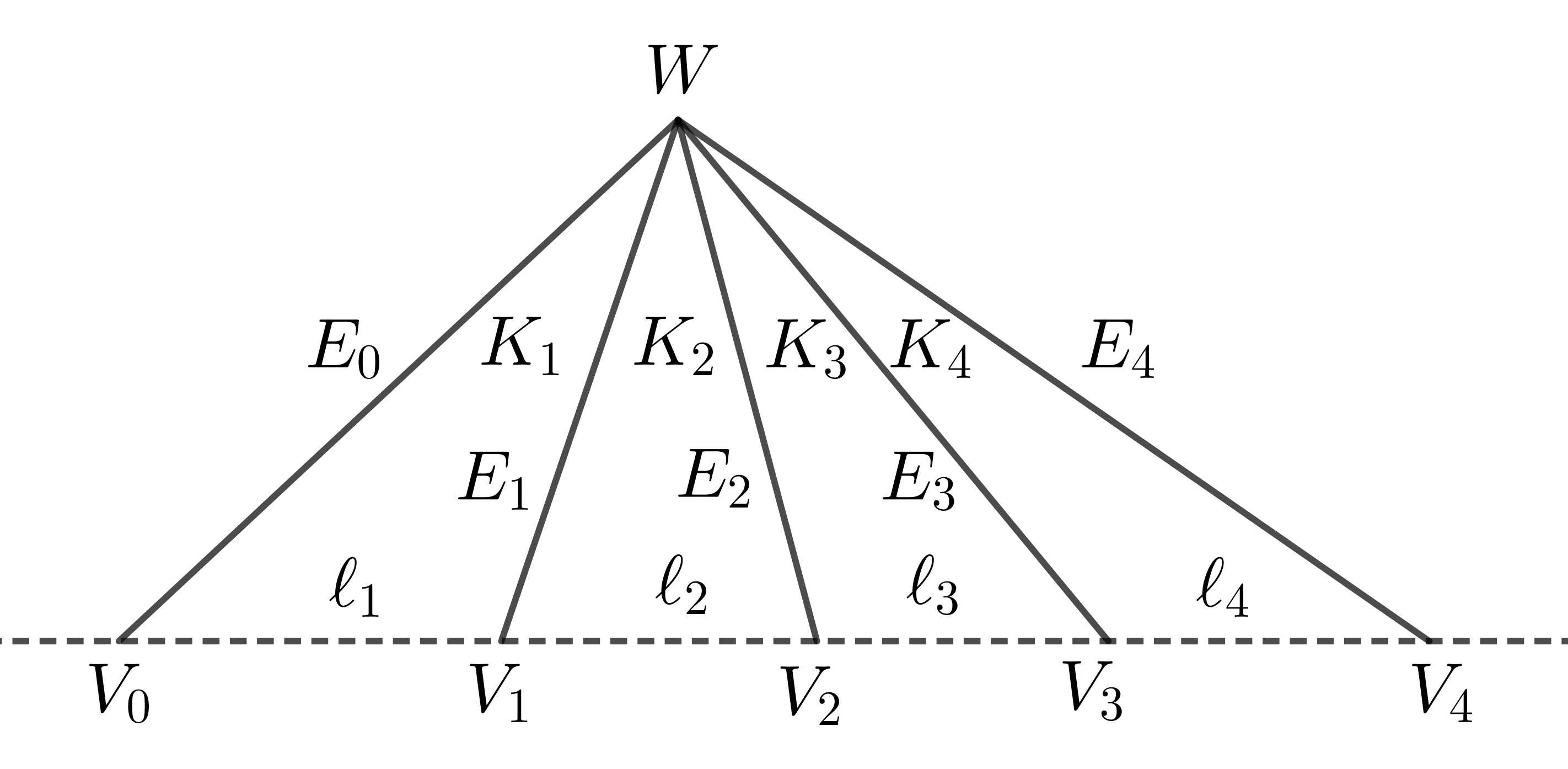}
\caption{Consecutive boundary singular  vertices $\V_1, \V_2, \V_3$}
\label{fig:singbdy}
\end{figure}

To avoid pathological meshes as the examples in Figure \ref{fig:pathomesh},
we assume the following on the triangulation $\Th$:
\begin{assumption}\label{asmesh}
  \begin{enumerate}
    \item
Each line segment of $\p\O$ connecting two corner of $\p\O$
has at least two regular vertices.
\item
  Each quasi singular vertex which is a corner  of $\p\O$
  has no interior edge connecting it to other boundary vertex.
\end{enumerate}
\end{assumption}
\begin{figure}[ht]
\hspace{1cm}
\subfloat[Each boundary segment has only one regular vertex.
]{
\includegraphics[width=0.28\textwidth]{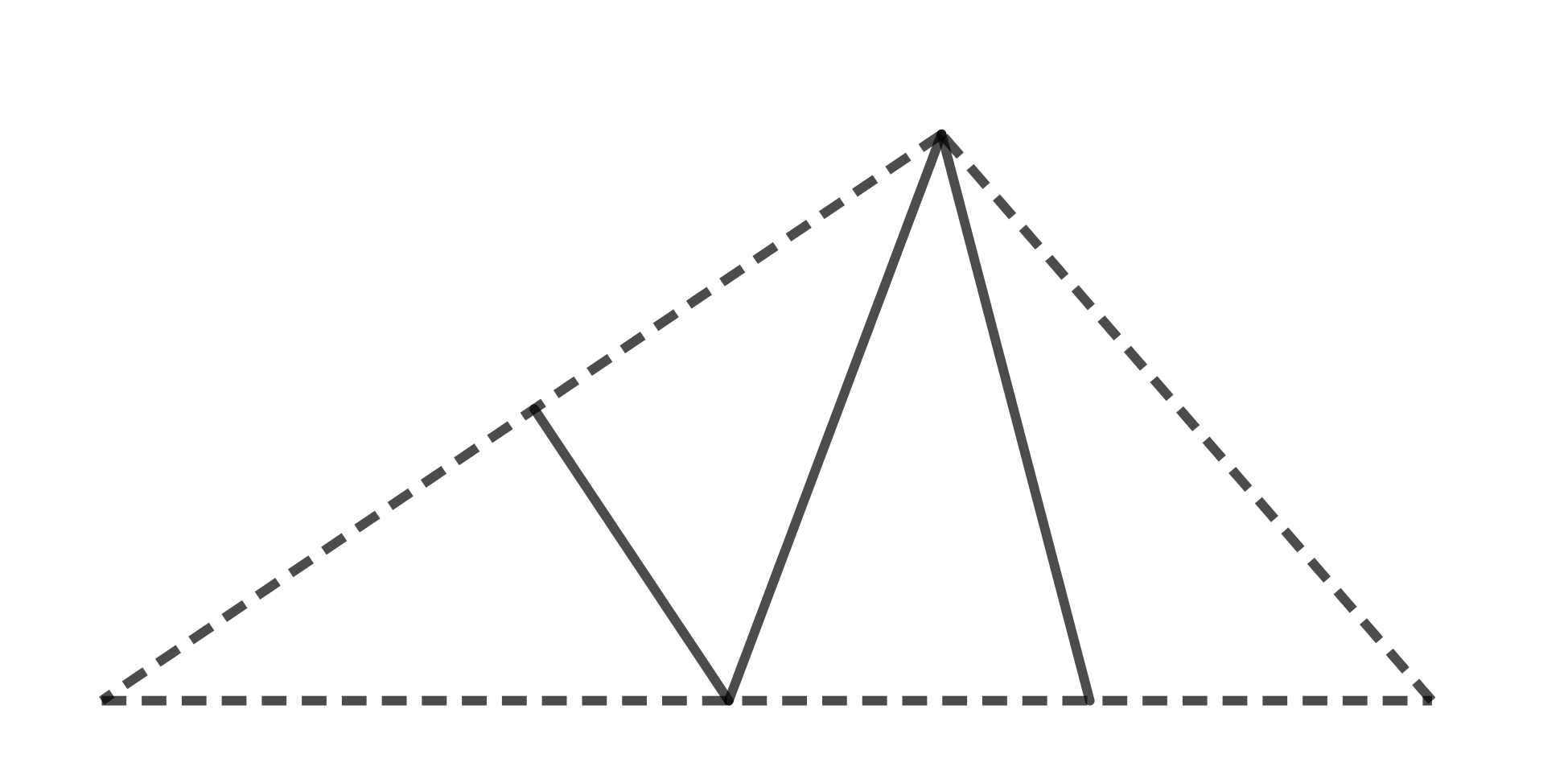}
}\quad
\subfloat[A quasi singular corner $\V$
is connected to other boundary vertex by an interior edge.
]{
\includegraphics[width=0.50\textwidth]{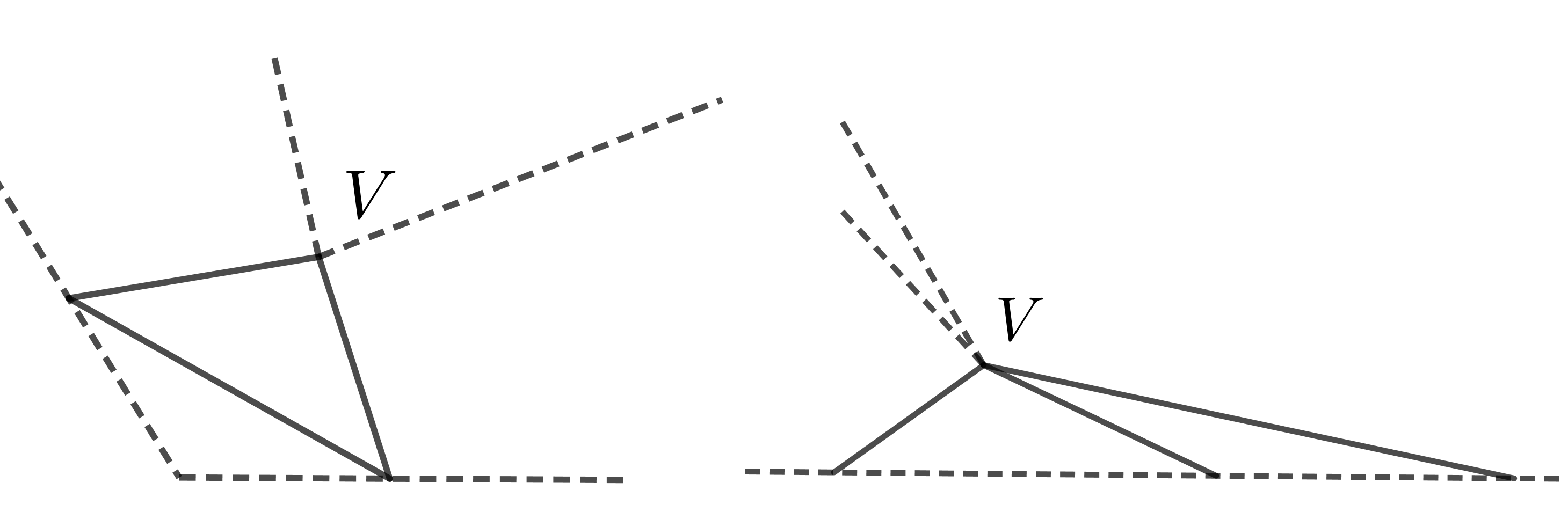}
}\qquad
\caption{Examples of pathological meshes (dashed lines belong to $\p\O$.)}
\label{fig:pathomesh}
\end{figure}

\subsubsection{Quasi singular chain not having any corner}\label{subsubsection:notcorner}
Let $\Q$ be a quasi singular chain which does not have any corner.
We can set in \eqref{post:def:S} that
$$\Q=\{ \V_1,\V_2,\cdots,\V_m \}\quad \mbox{ for } m\ge1,$$
and $\V_0, \V_{m+1}$ are regular vertices.

Then, we note that $\W$ is also regular.
It is clear  by Lemma \ref{lem:isolvtx}
if $\W$ is an interior vertex.
In case of $\W\in\p\O$,
$\W$ is not a corner as in Figure \ref{fig:pathomesh}-(b) by Assumption \ref{asmesh}.
Thus, $\W$ is regular on a line segment of $\p\O$ since $m\ge1$.

We can represent $e_h^{\V_k}$ with unknown constants $\alp_k, \bet_k$ as
\begin{equation}\label{post:reg:ehvk}
  e_h^{\V_k} =\alp_k \st_{E_{k-1}\V_k} + \bet_k\st_{E_{k+1}\V_k} \quad k=1,2,\cdots,m. 
\end{equation}
Then, by Lemma \ref{lem:c1pc2}, we have, for $k=1,2,\cdots,m$,
\begin{equation}\label{post:alpkellk}
  |\alp_k \ell_{k} + \bet_{k}\ell_{k+1}| \le
  \Cs(|\u-\u_h|_{1,\B(\V_k)} + \|p-\Pi_hp\|_{0,\B(\V_k)}).
\end{equation}

We note that $\V_1, \V_2,\cdots,\V_m$ are the only quasi singular vertices 
in $\B(\V_1, \V_2,\cdots,\V_m)$ since $\V_0,\V_{m+1},\W$ are regular. 
 Thus, from  \eqref{eq:spliteh}, \eqref{def:ehs}, \eqref{post:reg:ehvk}, we have
\begin{equation}\label{post:reg:ehKk}
  \begin{array}{c}
  e_h\Big|_{K_1} = (e_h^G+e_h^{SR})\Big|_{K_1}+\alp_1\st_{E_0\V_1},\quad
    e_h\Big|_{K_{m+1}} = (e_h^G+e_h^{SR})\Big|_{K_{m+1}}+\bet_m\st_{E_{m+1}\V_m}, \\
    e_h\Big|_{K_k} = (e_h^G+e_h^{SR})\Big|_{K_k}+
    \alp_k \st_{E_{k-1}\V_k} + \bet_{k-1}\st_{E_{k}\V_{k-1}},\quad k=2,3,\cdots,m.
  \end{array}
\end{equation}
Define a jump of a function $f$ at $\V_k$ as 
\[ [[f]]_{\V_k}=f|_{K_k}(\V_k)- f|_{K_{k+1}}(\V_k), \quad k=1,2,\cdots,m.\]
Then, from \eqref{post:reg:ehKk}, we have,   for $k=1,2,\cdots,m$,
\begin{equation}\label{eq:bdphjump}
  [[ p_h]]_{\V_k} =[[e_h]]_{\V_k} = (\alp_k-\frac1{10}\bet_{k-1})
  -(\bet_k-\frac1{10}\alp_{k+1})+[[e_h^G+e_h^{SR}]]_{\V_k},
\end{equation}
with the definition of sting functions in \eqref{def:zEV}.
In \eqref{eq:bdphjump}, $\bet_0=\alp_{m+1}=0$.

We can find $2m$ scalars
$\alpt_1, \bett_1, \alpt_2, \bett_2, \cdots,\alpt_m, \bett_m$ such that
\begin{equation}\label{post:reg:sys}
  \alpt_k \ell_{k} + \bett_{k}\ell_{k+1} =0,\quad
  [[ p_h]]_{\V_k} = (\alpt_k-\frac1{10}\bett_{k-1})
  -(\bett_k-\frac1{10}\alpt_{k+1}),\quad k=1,2,\cdots,m,
\end{equation}
where $\bett_{0}=\alpt_{m+1}=0$.
Note that the conditions in \eqref{post:reg:sys}
are similar to those in \eqref{post:alpkellk}, \eqref{eq:bdphjump}.
The existence of $\alpt_k,\bett_k$ is guaranteed by the argument
in the proof of Lemma \ref{lem:shVk} below.

Define discrete pressures $s_h^{\V_k}$ as
\begin{equation}\label{post:def:shVk}
  s_h^{\V_k} =\alpt_k \st_{E_{k-1}\V_k} + \bett_k\st_{E_{k+1}\V_k},
  \quad  k=1,2,\cdots,m.
\end{equation}
Then, the difference $e_h^{\V_k}-s_h^{\V_k}$ is estimated in the following lemma.
\begin{lemma}\label{lem:shVk}
  Let $\m$ be the mean of $e_h^G + e_h^{SR}$ over $\O$ Then, we have,
  for $k=1,2,\cdots,m$,
 \begin{equation}\label{eq:reg:ehvk}
   \norm{e_h^{\V_k} - s_h^{\V_k}}{0,\B(\V_k)}
\le \Cs (\norm{e_h^G + e_h^{SR}-\m}{0,\B(\W)} 
+ |\u-\u_h|_{1,\B(\W)} + \|p-\Pi_hp\|_{0,\B(W)}).
\end{equation}
\end{lemma}
\begin{proof}
  Let $\alpe_k=\alp_k-\alpt_k,\ \bete_k=\bet_k-\bett_k,\ k=1,2,\cdots,m$ and
  $\bete_0=\alpe_{m+1}=0$. Then
  from \eqref{post:alpkellk}-\eqref{post:reg:sys},
  we have
  \begin{equation}\label{post:reg:dalbe}
 \arraycolsep=1.4pt\def\arraystretch{2}   \begin{array}{c}
    |\alpe_k \ell_{k} + \bete_{k}\ell_{k+1}| \le
  \Cs(|\u-\u_h|_{1,\B(\V_k)} + \|p-\Pi_hp\|_{0,\B(\V_k)}), \\ 
(\alpe_k-\disp\frac1{10}\bete_{k-1})
      -(\bete_k-\disp\frac1{10}\alpe_{k+1})+[[e_h^S+e_h^{SR}]]_{\V_k}=0.
    \end{array}
  \end{equation}
Set $\aa_{m+1}=0$ and for $k=1,2,\cdots,m$,
\begin{equation}\label{post:reg:akbk}
  r_k=\frac{\ell_{k+1}}{\ell_{k}},\ \aa_k= \alpe_k+r_k\bete_k,\ 
\bb_k= \aa_k+\frac1{10}\aa_{k+1}  - (\alpe_k-\frac1{10}\bete_{k-1})
  +(\bete_k-\frac1{10}\alpe_{k+1}).
\end{equation}
Then, eliminating $\alpe_1,\alpe_2,\cdots,\alpe_{m+1}$ in \eqref{post:reg:akbk},
we have $m$ equations for $\bete_1,\bete_2,\cdots,\bete_{m}$,
\begin{equation}\label{post:sysalpb}
  \frac1{10}\bete_{k-1} + (1+r_k)\bete_k + \frac1{10}r_{k+1}\bete_{k+1}
  = b_k,\quad k=1,2,\cdots,m,  
\end{equation}
where $\bete_0=\bete_{m+1}=r_{m+1}=0$.

Rewrite \eqref{post:sysalpb}
with a matrix $A\in\R^{m\times m}$ in the form:
\begin{equation}\label{eq:Amatform}
  A(\bete_1,\bete_2,\cdots,\bete_m)^{t} = (\bb_1,\bb_2,\cdots,\bb_m).^t 
\end{equation}
For an example when $m=4$, since $\bete_0=\bete_{5}=0$,
\eqref{post:sysalpb} is written  in
\begin{equation}\label{eq:errormatrix}
  \left( \begin{array}{cccc}

           1+r_1 & r_2/{10} &  &  \\
           1/{10} & 1+ r_2 &  r_3/{10} & \\
                 &  1/{10} & 1+ r_3 &  r_4/{10} \\
          & & 1/{10} &  1+ r_4  \\
         \end{array}\right)
       \left( \begin{array}{c}
                \bete_1 \\
                \bete_2 \\
                \bete_3\\
                \bete_4
              \end{array}\right)
            = \left( \begin{array}{c}
                \bb_1 \\
                \bb_2 \\
                \bb_3\\
                \bb_4
              \end{array}\right).
          \end{equation}
 Note that $A$ is invertible since
 the transpose $A^t$ is strictly diagonally dominant.
 Thus, we have
            \begin{equation}\label{eq:normAm} \vertiii{A^{-1}}_2 \le \Cs,
            \end{equation}
since $m$ and $r_1,r_2,\cdots,r_m$  are bounded by $\Cs$.
From \eqref{post:reg:ehvk}, \eqref{post:def:shVk},
\eqref{post:reg:dalbe}, \eqref{post:reg:akbk}, \eqref{eq:Amatform}, \eqref{eq:normAm}, we obtain \eqref{eq:reg:ehvk} with
      $\aa_k=(\alpe_k \ell_{k} + \bete_{k}\ell_{k+1})/\ell_{k},\ k=1,2,\cdots,m.$       
\end{proof}

Now,
for each $\V\in \SS_h^{br}$, we can calculate $s_h^{\V}$
 similarly in \eqref{post:reg:sys}, \eqref{post:def:shVk}
and define
\begin{equation}\label{post:def:shbr}
   s_h^{br} = \sum_{ \V\in\SS_h^{br}} s_h^{\V}.
 \end{equation}
 Then, from Theorem \ref{thm:veloerror}, \ref{thm:ehGSR} and
Lemma \ref{lem:shVk}, we establish the following lemma.
 \begin{lemma}\label{lem:ehssbr}
   If $\u\in[H^5(\O)]^2, p\in H^4(\O)$, we have
   \begin{equation}\label{est:ehssbr}
    \norm{e_h^{SSbr}-s_h^{br}}{0} \le Ch^4(|\u|_{5}+ |p|_{4} ).
  \end{equation}
\end{lemma}

\subsubsection{Quasi singular chain having a corner}\label{subsubsection:corner}
Let $\hp_h = p_h - s_h^i -s_h^{br} $ and define
 \begin{equation}\label{post:def:ehb}
   \he_h= \hp_h -\Pi_h p = e_h^G + e_h^{SR} + e_h^{SSi}-s_h^i + e_h^{SSbr} -s_h^{br}
   + e_h^{SSbc}.
 \end{equation}
The remaining spurious error $ e_h^{SSbc}$ in \eqref{post:def:ehb} is our last target to be removed.
 
Let $\Q$ be a quasi singular chain containing  a corner $\C$
of two line segments $\Gamma, \Gamma_1$ of $\p\O$ 
such that
\begin{equation}\label{post:eq:Gamma01}
\#  (\Q\cap \Gamma_1)\ \le\ \# (\Q\cap \Gamma).
\end{equation}
We can set  in \eqref{post:def:S} that
$$\Q\cap \Gamma=\{ \V_0,\V_1,\V_2,\cdots,\V_m \}\quad \mbox{ for } m\ge0,$$
and $\V_0$ is the quasi singular corner $\C$ and  $\V_{m+1}$ is a regular vertex $\bR$.

Then, by Assumption \ref{asmesh}, $\V_{m+1}$ is not a corner. Thus,
there exists a triangle $K_{m+2}$ in $\Th$ which has the edge $E_{m+1}$ and
a vertex $X$ different to $\V_m$ as in Figure \ref{fig:cnrbwr}.

We remind that  $\SS_h^{bc}$ is the set of all
 boundary quasi singular  vertices consecutive from quasi singular corners.
 If $\W\in\SS_h^{bc}$, then $\W$ is quasi singular in $\Q\cap\Gamma_1$
 and $m=0$.
It contradicts to \eqref{post:eq:Gamma01}. 
Thus $\W\notin \SS_h^{bc}$.

For the vertex $\X$, if $\X\in\SS_h^{bc}$,
then $\W, \X$ lie on $\Gamma_1$ as in Figure \ref{fig:cnrbwr}-(a),
since $\X$ must be on a boundary line segment and $\W,\bR$ can not be corners
by Assumption \ref{asmesh}.
While $\W\notin \SS_h^{bc}$ is regular,  there exists one more regular vertex on $\Gamma_1$
by Assumption \ref{asmesh}, presented as $\bR_1$ in Figure \ref{fig:cnrbwr}-(a).
It conflicts with $\X\in\SS_h^{bc}$. Thus, we have  $\X\notin\SS_h^{bc}$, too.

With $2m+1$  unknown constants $\bet_0,\alp_1,\bet_1,\alp_2,\bet_2,\cdots,
\alp_m , \bet_m$,
we can represent that
\begin{equation}\label{post:cor:alpbet}
  e_h^{\C}\Big|_{K_1}=\bet_0\st_{E_1\C},\quad
  e_h^{\V_k} =\alp_k \st_{E_{k-1}\V_k} + \bet_k\st_{E_{k+1}\V_k} \quad k=1,2,\cdots,m. 
\end{equation}

\noindent
We note that $e_h^{SSbc}\big|_{K_{m+2}}=0$ since $\bR, \W, \X \notin \SS_h^{bc}$.
Thus, from \eqref{post:def:ehb}, \eqref{post:cor:alpbet}, we have 
\begin{equation}\label{post:cor:ehbkm}
  \arraycolsep=1.4pt\def\arraystretch{1.6}
  \begin{array}{lcl}
    \he_h|_{K_{m+2}}  &=&
      (e_h^G + e_h^{SR} + e_h^{SSi}-s_h^i + e_h^{SSbr} -s_h^{br})\Big|_{K_{m+2}},\\
    \he_h|_{K_{m+1}}  &=&
     (e_h^G + e_h^{SR} + e_h^{SSi}-s_h^i + e_h^{SSbr} -s_h^{br})\Big|_{K_{m+1}}
                          + \bet_m\st_{E_{m+1}\V_m},
  \end{array}
\end{equation}
and for $k=1,2,\cdots,m$,
\[\he_h\Big|_{K_k} = (e_h^G+e_h^{SR} + e_h^{SSi}-s_h^i + e_h^{SSbr} -s_h^{br}
                         )\Big|_{K_k}+ \alp_k \st_{E_{k-1}\V_k}+
    \bet_{k-1}\st_{E_{k}\V_{k-1}}.\]

Denote by $\M$, the midpoint of the edge $\overline{\bR\W}$ 
and define a jump of a function $f$ at $\M$ as
\[ [[f]]_{\M} = f\big|_{K_{m+1}} (\M) - f\big|_{K_{m+2}} (\M). \]
Then, from \eqref{post:cor:ehbkm}, we have 
\begin{equation}\label{post:cor:jumpm1}
  [[\hp_h]]_{\M}= [[ \he_h ]]_{\M} = -\frac1{10} \bet_m +
[[e_h^G + e_h^{SR} + e_h^{SSi}-s_h^i + e_h^{SSbr} -s_h^{br}]]_{\M},
\end{equation}
and for $k=1,2,\cdots,m$ and $\alp_{m+1}=0$, we have
\begin{equation}\label{eq:cor:bdphjump}
  [[ \hp_h]]_{\V_k} = (\alp_k-\frac1{10}\bet_{k-1})
  -(\bet_k-\frac1{10}\alp_{k+1})+[[e_h^G+e_h^{SR}+ e_h^{SSi}-s_h^i + e_h^{SSbr} -s_h^{br}]]_{\V_k}.
\end{equation}

As the unknowns satisfy
 \eqref{post:alpkellk}, \eqref{post:cor:jumpm1}, \eqref{eq:cor:bdphjump},
  we find $2m+1$ scalars $\bett_0,\alpt_1,\bett_1,\alpt_2,\bett_2,\cdots,
\alpt_m , \bett_m$ such that
\begin{equation}\label{def:sys:cornerab}
  [[\hp_h]]_{\M} = -\frac1{10} \bett_m,\quad
  [[ \hp_h]]_{\V_k} = (\alpt_k-\frac1{10}\bett_{k-1})-(\bett_k-\frac1{10}\alpt_{k+1}),\quad
 \alpt_k \ell_{k} + \bett_{k}\ell_{k+1} =0,
\end{equation}
with $\alpt_{m+1}=0$ and $k=1,2,\cdots,m$.
We can solve \eqref{def:sys:cornerab} by simple back substitution from $\bett_m$. 

Let $\m$ is the mean of $e_h^G + e_h^{SR}$ over $\O$ and denote
\begin{equation}\label{eq:cor:egh}
e_h^Z= e_h^G + e_h^{SR}-\m + e_h^{SSi}-s_h^i + e_h^{SSbr} -s_h^{br}.
\end{equation}
We can copy the notations and arguments in the proof of Lemma \ref{lem:shVk}
with removing $\bete_0=0$ and adding a equation for $\bete_{m}$ from
\eqref{post:cor:jumpm1}, \eqref{def:sys:cornerab}.
Then,
we meet a triangular system of $m+1$ linear equations
for $\bete_0, \bete_1, \cdots,\bete_m$
whose diagonal entries are all $1/10$.
Thus,
if we define discrete pressures $s_h^{\V_1},s_h^{\V_2},\cdots,s_h^{\V_m}$
as in \eqref{post:def:shVk},
the differences $e_h^{\V_k}-s_h^{\V_k}, k=1,2,\cdots,m$ satisfy
 \begin{equation}\label{eq:cor:ehvk}
   \norm{e_h^{\V_k} - s_h^{\V_k}}{0,\B(\V_k)}
\le \Cs (\|e_h^Z\|_{0,\B(\W)} 
+ |\u-\u_h|_{1,\B(\W)} + \|p-\Pi_hp\|_{0,\B(\W)}).
\end{equation}
In addition, we have
\begin{equation}\label{post:cor:bett0}
  \|(\bet_0-\bett_0)\st_{E_1\C}\|_{0,K_1} \le  \Cs ( \norm{e_h^Z}{0,\B(\W)}+
 |\u-\u_h|_{1,\B(\W)} + \|p-\Pi_hp\|_{0,\B(\W)}).
\end{equation}

Now, applying Lemma \ref{lem:c1pc2} and \eqref{post:cor:bett0}  with $\bett_0$
for every two back-to-back triangles in $\B(\C)$
in order starting at $K_1$,
we can find $s_h^{\C}$ consisting of sting functions such that
\begin{equation}\label{eq:cor:ehc}
 \norm{e_h^{\C}-s_h^{\C}}{0,\B(\C)}  \le \Cs (\|e_h^Z\|_{0,\B(\W,\C)}
+|\u-\u_h|_{1,\B(\W,\C)}
  + \|p-\Pi_hp\|_{0,\B(\W,\C)}).  
\end{equation}
Then for remaining vertices in $\Q\cap\Gamma_1\setminus\{\C\}=\{\Y_1,\Y_2,\cdots,\Y_n\}$ for $n\ge0$, utilizing similar jumps, 
we can find $s_h^{\Y_i}$ consisting of sting functions, $i=1,2,\cdots,n$ such that
\begin{equation}\label{eq:cor:ehy}
 \norm{e_h^{\Y_i}-s_h^{\Y_i}}{0,\B(\Y_i)}
  \le \Cs (\|e_h^Z\|_{0,\B(\Q)}+|\u-\u_h|_{1,\B(\Q)}+ \|p-\Pi_hp\|_{0,\B(\Q)}), 
\end{equation}
where $\B(\Q)=\B(\W,\C,\Y_1,\Y_2,\cdots,\Y_n)$.

After we have done this postprocess corner by corner of $\O$, we can define
\begin{equation}\label{post:def:shbc}
  s_h^{bc} = \sum_{ \V\in\SS_h^{bc}} s_h^{\V}.
\end{equation}
Then, combining  \eqref{eq:cor:egh}, \eqref{eq:cor:ehvk}, \eqref{eq:cor:ehc},
\eqref{eq:cor:ehy} with Theorem \ref{thm:veloerror}, \ref{thm:ehGSR}
and Lemma \ref{lem:errorbdyi}, \ref{lem:ehssbr},
we estimate that if $\u\in[H^5(\O)]^2, p\in H^4(\O)$,
\begin{equation}\label{post:est:ehssbc}
  \norm{e_h^{SSbc} - s_h^{bc}}{0} \le Ch^4(|\u|_{5}+ |p|_{4} ).
\end{equation}

\begin{figure}[ht]
\hspace{3mm}
\subfloat[$\W\in\p\O$]{
\includegraphics[width=0.43\textwidth]{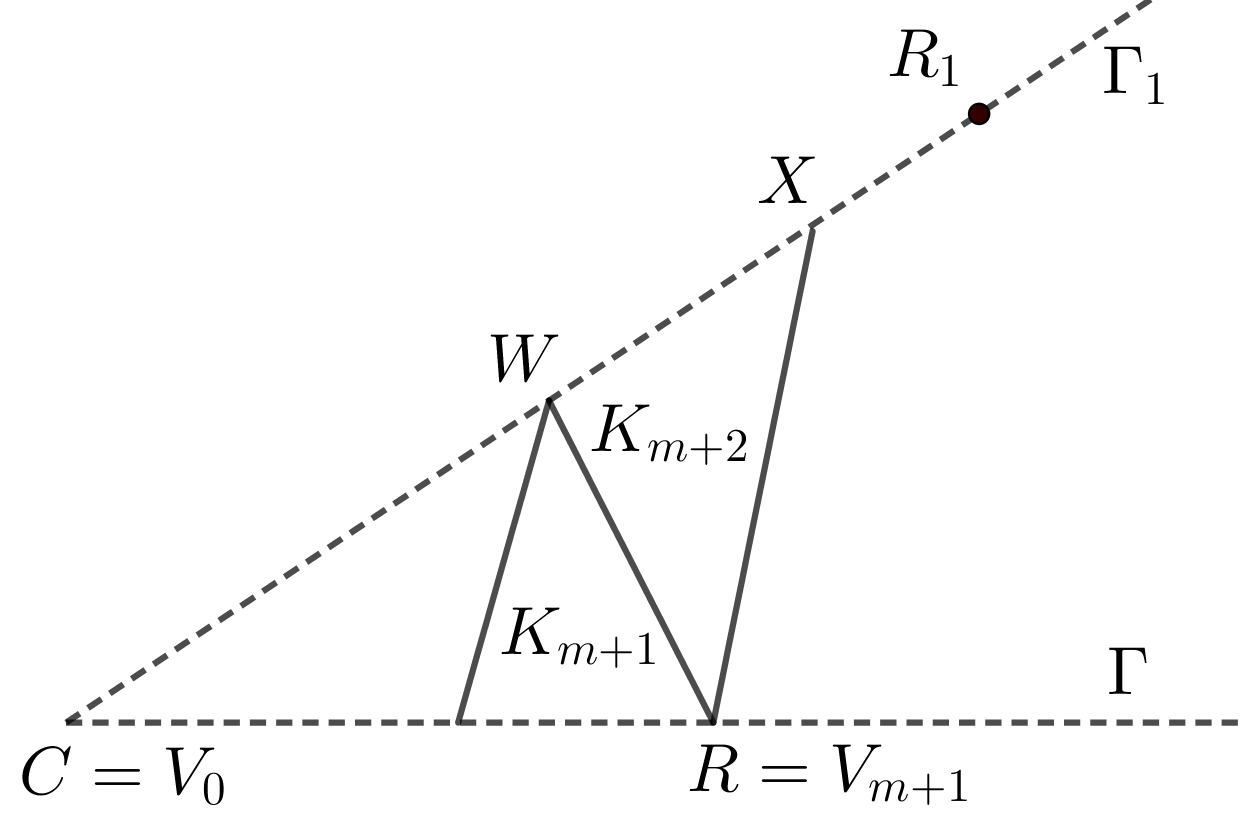}
}\quad
\subfloat[$\W\notin\p\O$]{
\includegraphics[width=0.45\textwidth]{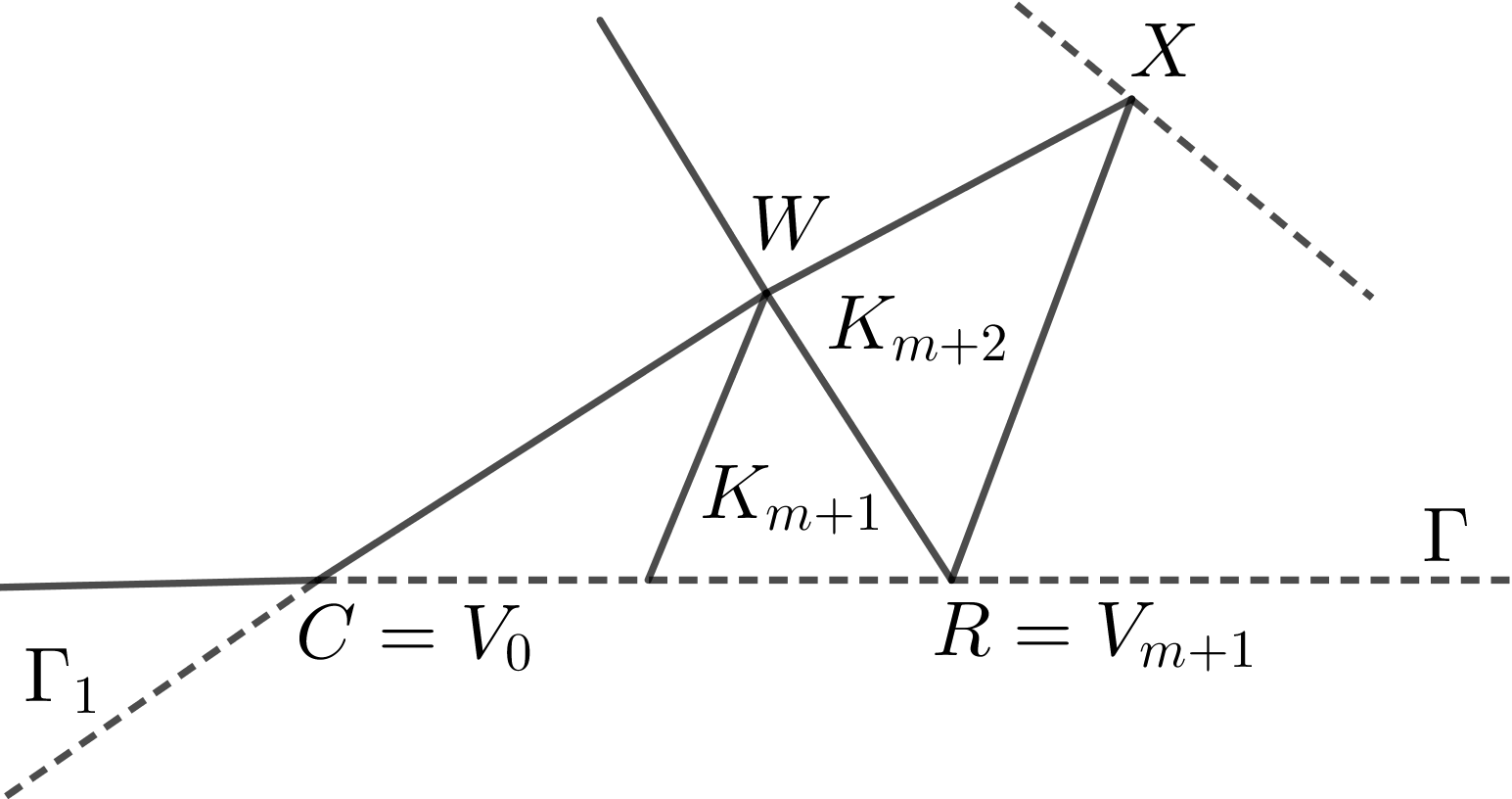}
}\caption{$\W, \X \notin \SS_h^{bc}$ and the  postprocessing starts at $K_{m+2}$.
 (dashed lines belong to $\p\O$.)}
\label{fig:cnrbwr}
\end{figure}

Now, we have calculated spurious pressures $s_h^i, s_h^{br}, s_h^{bc}$
 in \eqref{post:def:shi}, \eqref{post:def:shbr},  \eqref{post:def:shbc}.
Summing up them as
\begin{equation}\label{post:def:sh}
  s_h = s_h^i + s_h^{br} +s_h^{bc},
\end{equation}
and define $\tp_h\in M_h$ with the mean $\overline{s_h}$ of $s_h$ over $\O$ as
\begin{equation}\label{post:def:tph}
  \tp_h=p_h - (s_h - \overline{s_h}).
\end{equation}
Then, we reach at our final goal in the following theorem.
\begin{theorem}
  If $\u\in[H^5(\O)]^2, p\in H^4(\O)$, we have
  \begin{equation}\label{th:main}
    \norm{p-\tp_h }{0} \le C h^4(|\u|_{5} + |p|_4).
  \end{equation}
\end{theorem}
\begin{proof} From \eqref{def:eh}, \eqref{eq:spliteh}, \eqref{def:ehs},
  \eqref{eq:splitss}, \eqref{eq:splitssb},
  \eqref{post:def:sh}, \eqref{post:def:tph}, we have
  \begin{equation}\label{post:spantph}
    \tp_h-\Pi_hp = e_h - s_h^i - s_h^{br} -s_h^{bc} + \osh =   e_h^G+e_h^{SR}+e_h^{SSi}+e_h^{SSbr}+e_h^{SSbc} - s_h^i - s_h^{br} -s_h^{bc} + \osh.
  \end{equation}
  Let $\m_1, \m_2,\m_3,\m_4$ be means of $e_h^G+e_h^{SR}, e_h^{SSi}- s_h^i,
  e_h^{SSbr}- s_h^{br}, e_h^{SSbc} -s_h^{bc}$ over $\O$, respectively.
  Then, since  the mean of $\tp_h-\Pi_hp\in M_h$ over $\O$ vanishes, we have
$    \m_1 + \m_2 +\m_3 + \m_4+\osh =0$.

Thus, we can rewrite \eqref{post:spantph} into
 \begin{equation*}\label{post:spantph2}
    \tp_h-\Pi_hp =  (e_h^G+e_h^{SR}-\m_1 ) + (e_h^{SSi}- s_h^i -\m_2)
    + (e_h^{SSbr}- s_h^{br}-\m_3) + (e_h^{SSbc} -s_h^{bc}-\m_4),
  \end{equation*}
  which establishes \eqref{th:main} by
\eqref{eq:mainstable}, \eqref{est:ehssbr},   \eqref{post:est:ehssbc}
and 
  Theorem \ref{thm:ehGSR}.
\end{proof}

\section{Numerical results}
We did numerical experiments in $\O=[0,1]^2$ with the velocity $\u$
and pressure $p$ such that
\[ \u=(s(x)s'(y), -s'(x)s(y)),\quad \ p=\sin(4\pi x)e^{\pi y},   \]
where $ s(t)=(t^2-t)\sin(2\pi t)$.

For triangulations with quasi singular vertices, we first formed the meshes of $\O$
with uniform squares and added a quasi singular vertex $\V$
in every squares so that $\V$ divides the diagonal of positive slope
with ratio $1.0005:1$ as in Figure \ref{fig:mesh}-(b).
An example of $8\times 8\times 4$ mesh is depicted in Figure \ref{fig:mesh}-(a).

A direct linear solver in {\texttt {LAPACK}}
was used on solving the discrete Stokes problem \eqref{prob:discr}.
Then, 
as in Figure \ref{fig:pandph}, the discrete pressure $p_h$ is spoiled by
spurious error at a glance. A closer look 
over 4 triangles in Figure \ref{fig:pandph2} shows the alternating characteristic of spurious error,
as predicted in \eqref{eq:spuqh}.

The postprocessed $\tp_h$ from $p_h$
shows that the spurious error in $p_h$
is removed as in Figure \ref{fig:postph}.
The errors in $\tp_h$ are also much less than those in $p_h$
as listed in Table \ref{table}.
Even in case of regular vertices as in Figure \ref{fig:mesh35},
the postprocessed  $\tp_h$ improved the error in pressure as in Table \ref{tabler}.

\begin{figure}[ht]
  \hspace{15mm}
  \subfloat[$\Th$ : $8\times 8\times 4$ mesh]
{\includegraphics[width=0.35\textwidth]{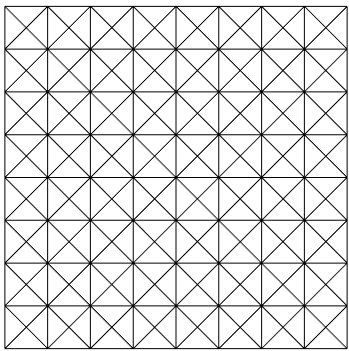}}
\qquad\qquad\qquad
\subfloat[$a:b=1.0005:1$]{\raisebox{6ex}
{\includegraphics[width=0.2\textwidth]{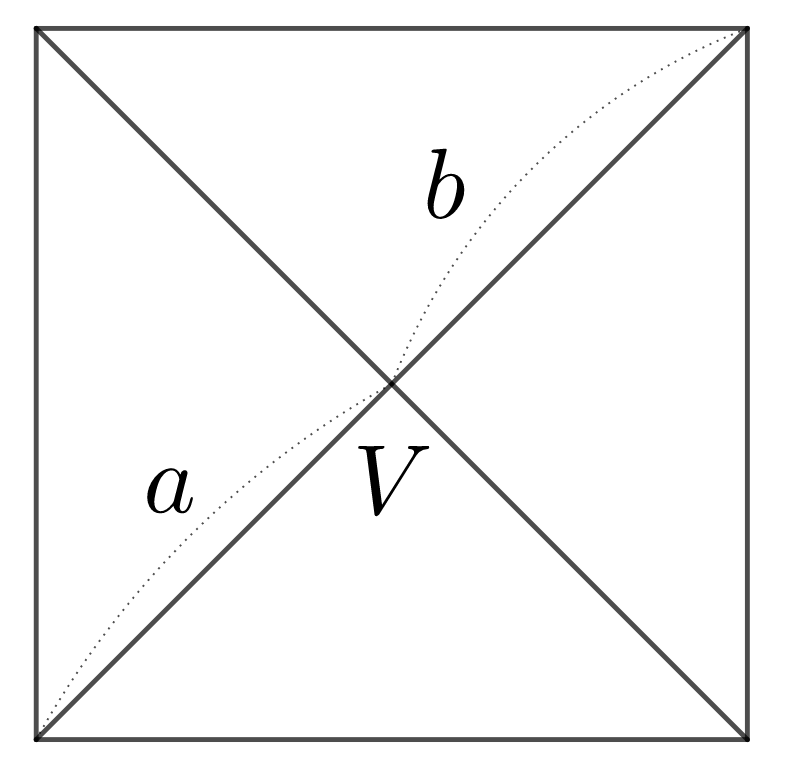}}}
\caption{An example of $\Th$ with
  a quasi singular vertex $\V$ in every squares}
  \label{fig:mesh}
\end{figure}

\begin{figure}[ht]
 \subfloat[$p$ over $\O$]
 {\includegraphics[width=0.47\textwidth]{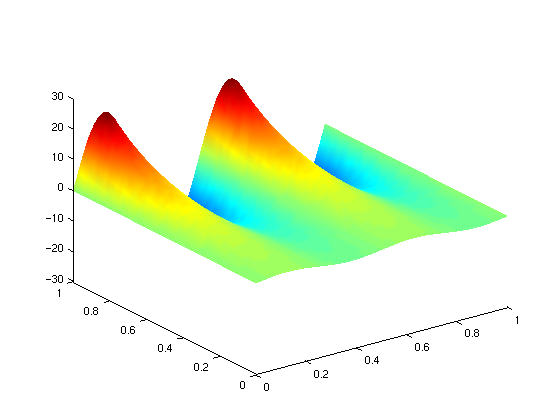}}
 \qquad
  \subfloat[$p_h$  over $\O$ ]
  {\includegraphics[width=0.47\textwidth]{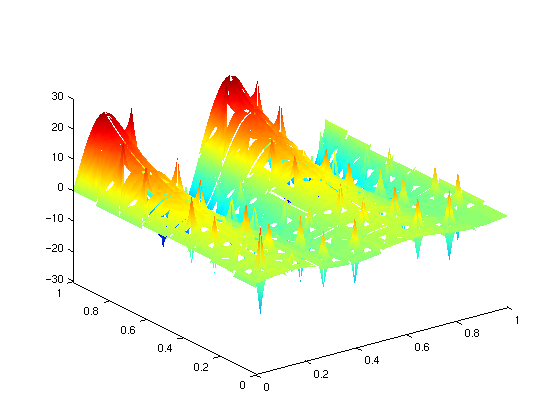}}
  \caption{Graphs of $p$ and $p_h$
    solved in $8\times 8\times 4$ mesh in Figure \ref{fig:mesh}}
  \label{fig:pandph}
\end{figure}

\begin{figure}[ht]
 \subfloat[$p$ over 4 triangles]
 {\includegraphics[width=0.47\textwidth]{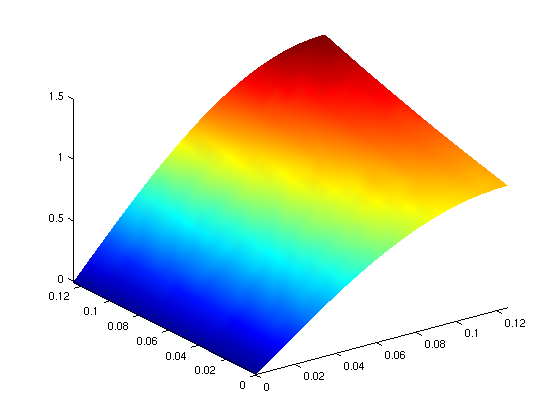}}
 \qquad
  \subfloat[$p_h$ over 4 triangles]
  {\includegraphics[width=0.47\textwidth]{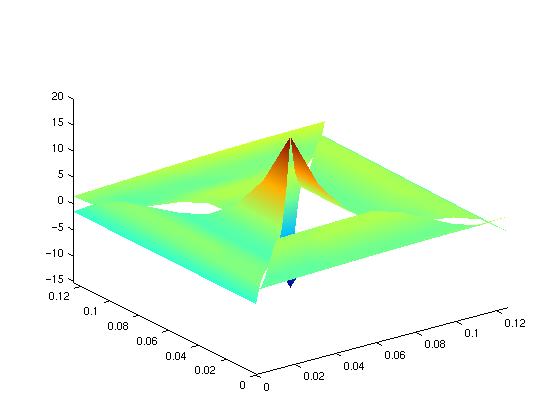}}
  \caption{Graphs of $p$ and $p_h$ over 4 triangles in $[0,1/8]^2$ }
  \label{fig:pandph2}
\end{figure}

\begin{figure}[ht]
 \subfloat[$\tp_h$ over $\O$]
 {\includegraphics[width=0.47\textwidth]{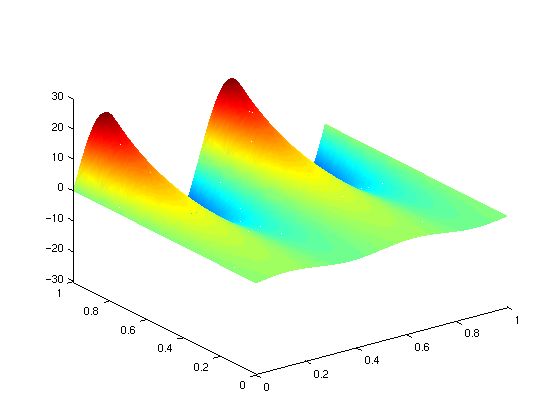}}
 \qquad
   \subfloat[$\tp_h$ over 4 triangles]
  {\includegraphics[width=0.47\textwidth]{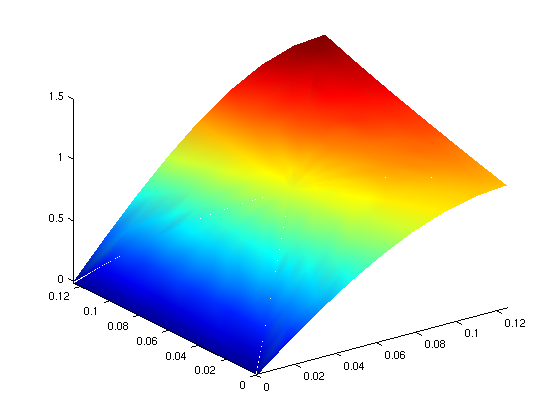}}
  \caption{Graphs of postprocessed $\tp_h$ from $p_h$}
  \label{fig:postph}
\end{figure}

\begin{table}
  \hspace{8mm}
\begin{tabular}{r || c c || c c || c c}
\hline
  mesh\hspace{3mm}
  & \hspace{1mm}$|\u-\u_h|_1$\hspace{1mm} & order\hspace{1mm}
     & \hspace{1mm}$\|p-p_h\|_0$\hspace{1mm}
  & order\hspace{1mm}& \hspace{1mm}$\|p-\tp_h\|_0$\hspace{1mm} & order \\
\hline
4 x 4 x 4 & 8.5504E-3 & & 2.2102E+1 & &5.7023E-2  & \\
 8 x 8 x 4 & 5.4471E-4 & 3.9724 &8.3012E-1 &  4.7347 & 2.6680E-3  & 4.4177\\
 16 x 16 x 4 & 3.3925E-5 & 4.0051 &2.6856E-2  & 4.9500 & 1.6624E-4  & 4.0044 \\
 32 x 32 x 4 & 2.1182E-6 &4.0014 &9.8863E-4  & 4.7637  &1.0380E-5  & 4.0014 \\
\hline
\end{tabular}
\caption{\label{table}Error table for meshes with quasi singular vertices
  as in Figure \ref{fig:mesh}}
\end{table}

\begin{figure}[ht]
  \hspace{15mm}
  \subfloat[$\Th$ : $8\times 8\times 4$ mesh]
{\includegraphics[width=0.35\textwidth]{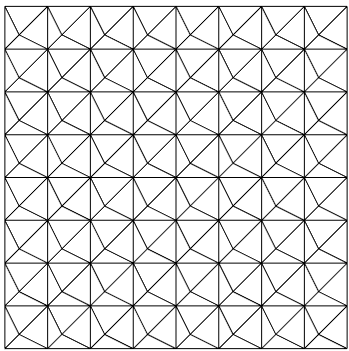}}
\qquad\qquad\qquad
\subfloat[$a:b=3:5$]{\raisebox{6ex}
{\includegraphics[width=0.2\textwidth]{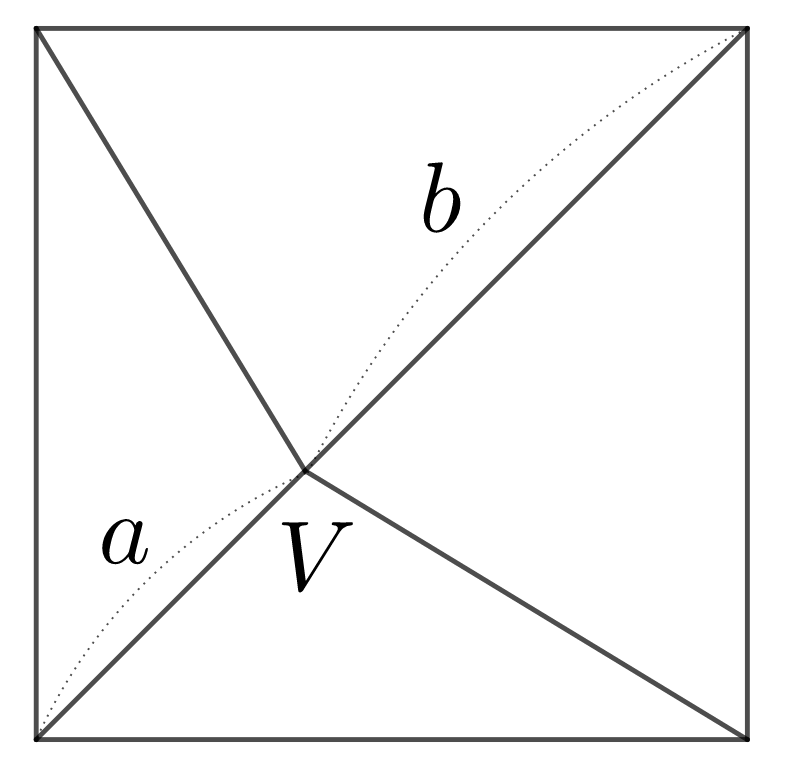}}}
\caption{$\Th$ with no quasi singular vertex}
  \label{fig:mesh35}
\end{figure}

\begin{table}
  \hspace{8mm}
\begin{tabular}{r || c c || c c || c c}
\hline
  mesh\hspace{3mm}
  & \hspace{1mm}$|\u-\u_h|_1$\hspace{1mm} & order\hspace{1mm}
     & \hspace{1mm}$\|p-p_h\|_0$\hspace{1mm}
  & order\hspace{1mm}& \hspace{1mm}$\|p-\tp_h\|_0$\hspace{1mm} & order \\
\hline
4 x 4 x 4 &1.3539E-2 & &9.8341E-2 & & 6.9479E-2  & \\
 8 x 8 x 4 &8.7627E-4 &3.9496 &5.6435E-3 &4.1231 & 3.4819E-3  &4.3186 \\
 16 x 16 x 4 &5.4353E-5 &4.0109 &3.4576E-4 &4.0287 &2.1298E-4  &4.0311 \\
 32 x 32 x 4 &3.3688E-6 &4.0120 &2.1285E-5 &4.0218 &1.3114E-5  &  4.0216\\
\hline
\end{tabular}
\caption{\label{tabler}Error table for meshes with no quasi singular vertex
  as in Figure \ref{fig:mesh35}}
\end{table}

\section*{Acknowledgments}
This paper was supported by Konkuk University in 2015.

\end{document}